\documentclass[12pt,a4paper,reqno]{amsart}            
\usepackage[utf8]{inputenc}
\usepackage{a4wide}

\title[Sparse domination in nonhomogeneous trees]{Endpoint estimates and sparse domination in nonhomogeneous trees}

\author[Conde Alonso, De Mari, Monti, Rizzo, Vallarino]{Jos\'e M. Conde Alonso , Filippo De Mari, Matteo Monti, Elena Rizzo, Maria Vallarino}

\thanks{Conde-Alonso was partially supported by grants \texttt{CNS2022-135431} and \texttt{RYC2019-027910-I} (Ministerio de Ciencia, Spain). 
De Mari, Monti, Rizzo and Vallarino were partially supported by the Project ``Harmonic analysis on continuous and discrete structures'' (bando Trapezio Compagnia di San Paolo CUP \texttt{E13C21000270007}) and are members of the Gruppo Nazionale per l'Analisi Matematica, la Probabilit\`a e le loro Applicazioni (GNAMPA) of the Istituto Nazionale di Alta Matematica (INdAM). Monti is supported by the PRIN 2022 project “TIGRECO—TIme-varying signals on Graphs: REal and COmplex methods” funded by the European Union—Next Generation EU (grant \texttt{20227TRY8H}, CUP \texttt{F53D23002630001}).
}

\usepackage{amsthm,amssymb,amsmath,amsopn,amsfonts,pb-diagram,accents,xcolor,array,hyperref,mathtools,enumerate,enumitem,bm,bbm,mathrsfs,amscd}
\usepackage{authblk}
\usepackage{comment}
\usepackage{algpseudocode}
\usepackage{multicol}
\usepackage{mathtools}
\usepackage[subnum]{cases}

\newtheorem{theorem}{Theorem}[section]
\newtheorem{lemma}[theorem]{Lemma}
\theoremstyle{definition}

\newtheorem{prop}[theorem]{Proposition}

\numberwithin{equation}{section}

\newtheorem{ltheorem}{Theorem}

\theoremstyle{remark}
\newtheorem{remark}[theorem]{Remark}

\allowdisplaybreaks
\begin{document}


\newcommand{\N}{\mathbb{N}}
\newcommand{\Z}{\mathbb{Z}}
\newcommand{\R}{\mathbb{R}}
\newcommand{\C}{\mathbb{C}}


\newcommand{\D}{\mathscr{D}}
\newcommand{\M}{\mathcal{M}}
\newcommand{\avg}[2]{\langle{#1}\rangle_{#2}}


\newcommand{\Tr}{\mathcal{X}}
\newcommand{\orig}{0}
\newcommand{\qf}{q}
\newcommand{\qfint}{\tilde{\qf}}
\newcommand{\Sec}{S}
\newcommand{\graphdist}{d}
\newcommand{\Grdist}{\mathrm{Gr}}
\newcommand{\sparsedist}{\mathfrak{d}}
\newcommand{\sons}[1]{\D_1(#1)}
\newcommand{\cl}[1]{\overline{#1}}
\newcommand{\anc}[2]{{#1}^{(#2)}}
\newcommand{\Or}[1]{x_{#1}}


\newcommand{\one}{1}
\newcommand{\Ss}{\mathcal{S}}
\newcommand{\As}{\mathcal{A}_{\Ss}} 
\newcommand{\Extras}{\mathcal{E}_{\Ss}} 


\newcommand{\BMO}{\mathrm{BMO}}
\newcommand{\Hone}{\mathrm{H}^1}
\newcommand{\Berg}{\mathcal{A}}
\newcommand{\cexp}{\mathsf{E}}
\newcommand{\martdiff}{\mathsf{D}}


\newcommand{\RH}{\varphi}


\newcommand{\Proj}{\mathcal{P}}
\newcommand{\K}{\mathcal{K}} 
\newcommand{\h}{h} 
\newcommand{\Diff}{\mathcal{D}} 
\newcommand{\Prest}{\Proj_{\mathcal{F}^1}^{Q_0}}
\newcommand{\PrestF}{\Proj_{\mathcal{F}}^{Q_0}}
\newcommand{\Krest}{\K_{\mathcal{F}^1}^{Q_0}}
\newcommand{\KrestF}{\K_{\mathcal F}^{Q_0}}

\newcommand{\BB}{\mathcal{B}} 
\newcommand{\tBB}{\tilde{\BB}}
\newcommand{\Muck}{{\BB_2(\mu)}}
\newcommand{\ME}{{\BB_2^D(\mu)}}


\newcommand{\CZ}{Calder\'on-Zygmund }
\newcommand{\supp}{{\rm supp}}

\maketitle

\begin{abstract}
We prove endpoint and sparse-like bounds for Bergman projectors on nonhomogeneous, radial trees $\Tr$ that model manifolds with possibly unbounded geometry. The natural Bergman measures on $\Tr$ may fail to be doubling, and even locally doubling, with respect to the right metric in our setting. Weighted consequences of our sparse domination results are also considered, and are in line with the known results in the disk. Our endpoint results are partly a consequence of a new Calder\'on-Zygmund theory for discrete, non-locally doubling metric spaces. 
\end{abstract}


\section*{Introduction}

Discrete models of continuous objects are omnipresent in harmonic analysis. Many operators of interest can often be difficult to study directly due to issues related to convergence in their definition, subtle singular behavior, or other measure theoretic/functional analytic complications. Discrete models can remedy all these by providing objects which are easier to handle and that are defined and act on a measure space that can be identified with $\N$. This allows one to isolate the questions that one is interested in answering from the technicalities that surround them. Examples of successful discrete models include Haar shifts and other dyadic objects like square functions or martingale transforms (see \cite{Pe2019} for a general account or \cite{Hy2012} for an outstanding result), wavelet transforms \cite{DPWW2024}, or weighted graphs that model weighted manifolds \cite{KLW}.

We are interested in discrete models of the hyperbolic disk and other noncompact symmetric manifolds. The natural discrete structure to use in that case is that of infinite trees. Trees are connected, loop-free graphs which may have a distinguished vertex $\orig$. They are called homogeneous when all vertices have the same number of neighbors. It is known that homogeneous trees can be thought of as a discrete analogue of the hyperbolic disk, as made precise in, for example, \cite{CC1994}. In the same spirit, more general trees can be thought of as discrete counterparts of more general Riemannian manifolds. In this work, we consider trees $\Tr$ that are radial. We say that $\Tr$ is radial when each vertex $x$ has exactly $\qf(x)+1$ neighbors in $\Tr$, and there exists $\qfint:\Z_+ \to \Z_+$ such that
$$
\qf(x) = \qfint(|x|) \quad \mbox{for all} \; x \in \Tr,
$$
where $|x|$ denotes the distance from $x$ to $\orig$. Therefore, all vertices $x$ at the same distance from $\orig$ have the same number of neighbors. Radial trees can be seen as analogues of noncompact rotationally symmetric Riemannian manifolds (see \cite[Section 2.3]{Pet}), and so our results might shed light on their analysis at infinity. 

Radial trees $\Tr$ can be equipped with a measure from a natural class. Fix $\alpha >1$. We consider a measure on $\Tr$ given by 
$$
\mu(\{x\}) = \prod_{\ell=0}^{|x|-1} \frac{1}{\qfint(\ell)^\alpha}, \quad x \in \Tr.
$$
We assume $\qfint(\ell) \geq 2$ for all $\ell$, so $\mu$ is a finite measure. The measure spaces that one obtains by making $\alpha=0$ or $\alpha=1$, which do not fall under the scope of our work, correspond to the well known cases of the counting measure and the canonical flow measure, respectively. The case of $\alpha>1$ has also been considered before, albeit most often in a special situation described above, the homogeneous one. The measures introduced above are indeed a generalization of a class of the measures introduced in \cite{CCPS}, where harmonic Bergman spaces on homogeneous trees were first considered. In turn, those are the natural discrete counterpart of the natural Bergman measures in the continuous hyperbolic disk, with the case of $\alpha=2$ corresponding to the Lebesgue measure.

We aim to understand the behavior of certain integral operators on $(\Tr,\mu)$ and its associated $L^p$-spaces. An outstanding example of such an operator of interest in the hyperbolic disk |and other hyperbolic manifolds| is the Bergman projector $\Proj$, the projection from $L^2$ to the subspace of holomorphic functions. In our real variable, discrete context, $\Proj$ is the orthogonal projection from $L^2(\Tr)$ to the Bergman space $\Berg^2(\Tr)$, the subspace of functions which are harmonic with respect to the combinatorial Laplacian on $\Tr$. We are interested in the $L^p$ and endpoint behavior of $\Proj$. In addition, understanding weighted inequalities for $\Proj$ is important in the study of Toeplitz operators in the continuous case, and we shall make a contribution in that direction in the discrete case. The work \cite{dMMV2023} already considers $L^p$-boundedness for $\Proj$ in the homogeneous case, but weighted inequalities do not seem to follow easily from the approach taken there. In addition, our mild assumptions on $\Tr$ imply that its geometry may be unbounded. In terms of harmonic analysis, this means that the measure $\mu$ is not only nondoubling with respect to the natural graph metric, but it is also nondoubling |even non-locally doubling| with respect to the Gromov metric \cite[p. 245]{Wo}, which is the right one to study $\Proj$ as indicated by the results in \cite{Mo}. The lack of the doubling property of $\mu$ is one of the main challenges that we have to overcome in our setting.

The last few years have seen the development of a new technique within harmonic analysis that entails (usually sharp) weighted inequalities: sparse domination. Initially developed to provide a simpler proof of the so-called $A_2$ theorem \cite{Le2013,Le2013b}, it has later evolved to become a widespread tool to prove sharp estimates for many operators of interest. One can consult the survey \cite{Pe2019} and the references therein for an extensive account of the developments in the area. A collection of sets $\Ss$ is sparse if for each $Q \in \Ss$ there exists $E_Q \subseteq Q$ such that $\mu(E_Q) \geq \frac{1}{2} \mu(Q)$, and the family $\{E_Q: Q \in \Ss\}$ is pairwise disjoint. The sparse operator associated with $\Ss$ is an averaging operator
$$
\As f(x) = \sum_{Q\in \Ss} \langle f \rangle_Q \one_Q(x),
$$
where $\avg{f}{Q}$ denotes the integral average of $f$ over the set $Q$. A sparse domination result for an operator $T$ |in the dual sense| is an inequality like
$$
\langle Tf_1,f_2\rangle \lesssim \langle \As f_1, f_2 \rangle,
$$
where the sparse collection of sets $\Ss$ may depend on $f_1$ and $f_2$. We intend to follow that trend and look for an appropriate sparse domination result for $\Proj$. Like we said, the main obstacle that we shall encounter is that, even with the right choice of metric structure on $\Tr$
, $\mu$ fails to be doubling over the relevant family of sets that are candidates to belong to a sparse family $\Ss$. In the nondoubling context, the existing sparse literature is either unsatisfactory \cite{CoPa2019,VoZK2018} or requires a substantial modification of the concept of sparse form \cite{CPW2023}. This modification consists in adding an extra bilinear form to the right hand side of the inequality which accounts for certain off-diagonal interactions between the two functions $f_1$ and $f_2$. Despite the fact that this extra form cannot be bounded by a sparse one |for the original operator does not admit sparse domination itself|, one can still recover useful quantitative information, including $L^p$ boundedness results, from the resulting domination. We follow this and term our extra form $\Extras$.

\begin{ltheorem} \label{th:theoremA}
Let $f_1,f_2 \in L^1(\Tr)$. There exists a sparse collection $\Ss$ such that
$$
|\langle \Proj f_1,f_2 \rangle| \lesssim \langle \As f_1, f_2 \rangle + \Extras(f_1,f_2).
$$
\end{ltheorem}

We postpone the definition of $\Extras$ to Section \ref{sec2} (see formula \eqref{eq:sparsenew}). It captures the local behavior of the operator, accounting for interactions between vertices in $\Tr$ |and hence sets in $\Ss$| that are close to one another. $\Extras$ cannot be avoided in the general case, but it is not needed when $\qf$ is a bounded function. In that case, our argument can be easily modified to yield a domination of $\Proj$ involving only a usual sparse form. It should be noted that in that case the result can be achieved using one of the existing sparse domination principles, like that in \cite{Lo2021}. However, as we already explained, the nondoubling context poses significant difficulties making all previously existing approaches unapplicable. In order to prove Theorem \ref{th:theoremA}, we avoid writing down a explicit formula for the kernel of $\Proj$, and instead we look at it as a variant of a Haar shift. The analogy is far from perfect and we still need to account for tail issues and lack of localization, but it allows us to loosely follow an approach reminiscent of the beautiful proof in \cite{CuDPOu2018}. 

Our sparse domination result for $\Proj$ is of course tailored to the operator under analysis. However, a more general principle lies behind it that can be applied to any discrete nondoubling setting, where small sets and local singularities do not play any role. Indeed, given a general linear operator $T$, our sparse domination procedure can always be carried out, leading to the identification of the right analogue of $\Extras$ if it is needed. One can see the dyadic results of \cite{CPW2023} and Theorem \ref{th:theoremA} above as instances of this general strategy, of which more details are given in Section \ref{sec2}.

We next identify the weighted inequalities for $\Proj$ that follow from our sparse domination result. As we explained above, we are forced to include the form $\Extras$ as part of our dominating term. This enters the definition of our classes of weights: the right one to dominate sparse operators is the Muckenhaupt class, while we necessarily need a larger one. For each $1<p<\infty$, we denote the Muckenhaupt-like class of weights associated to the sets that form the families $\Ss$ by $\BB_p(\mu)$. Additionally, a |generally larger| quantity $[w]_{\tBB_p(\mu)}$ is introduced in order to bound the term $\Extras$ of the sparse domination. Then, our weighted inequalities for $\Proj$ read as follows:   

\begin{ltheorem} \label{th:theoremB}
Let $1<p<\infty$. If $w \in \tBB_p(\mu)$, then 
$$
\|\Proj f\|_{L^p(\Tr,w)} \lesssim_{[w]_{\tBB_p(\mu)}} \| f\|_{L^p(\Tr,w)}.
$$
\end{ltheorem}

The dependence on $[w]_{\tBB_p(\mu)}$ in the inequality above can be traced. For example, if $p=2$ one gets $[w]_{\BB_2(\mu)}[w]_{\tBB_2(\mu)}^{1/2}$, which does not reduce to the optimal one in the classical case, as we explain below. In our proof, the extra power comes from the estimate of the term $\Extras$ and does not seem to be avoidable in the nondoubling setting. We have decided to term $\tBB_p(\mu)$ the class of weights that appears in Theorem \ref{th:theoremB}, because the class of sets over which weights $w$ are tested in the definition of $[w]_{\BB_p(\mu)}$ makes it the exact discrete analogue of the Bekoll\'e-Bonami class of weights \cite{BB1978}, which is known to characterize the weighted bounds for the Bergman projection on the disk. In addition, our quantitative results are in line with their continuous counterpart from \cite{PR2013} |in the doubling setting|. The proof of Theorem \ref{th:theoremB} follows the standard argument that deduces weighted bounds from sparse domination (see, for instance, \cite{CMP2012}), that we adapt to deal with the new term $\Extras$ in our sparse domination result. We emphasize again that our scheme of proof for Theorems \ref{th:theoremA} and \ref{th:theoremB} has a wider applicability: we identify conditions under which an operator defined on a discrete structure that is equipped with a martingale filtration admits a sparse domination result and a family of nontrivial weighted inequalities. 

We finally study the endpoint behavior of $\Proj$. At this point, we prefer to formulate our results in two separate ways. First, we study the endpoint behavior of general integral operators $T$ whose kernels $K$ satisfy adequate H\"ormander-like and integral size conditions, that we term respectively \eqref{eq:HormCond} and \eqref{eq:sizeCond}. These should be compared to those naturally arising in nonhomogeneous Calder\'on-Zygmund theory \cite{To2001,To2001b}. Our results hold for any tree equipped with any measure, and can be thought of as a general \CZ theory for non-locally doubling trees. 

\begin{ltheorem} \label{th:theoremC}
    Let $(\Tr,\mu)$ be a tree equipped with any measure. Let $T$ be an $L^2$-bounded integral operator with symmetric kernel $K$. If $K$ satisfies conditions \eqref{eq:sizeCond} and \eqref{eq:HormCond}, then $T$ satisfies the following endpoint bounds:
    \begin{enumerate}
        \item[(a)] $T$ maps $L^1(\Tr)$ to $L^{1,\infty}(\Tr)$.
        \item[(b)] $T$ maps $\Hone(\Tr)$ to $L^1(\Tr)$.
        \item[(c)] $T$ maps $L^\infty(\Tr)$ to $\BMO(\Tr)$.
    \end{enumerate}
\end{ltheorem}

The spaces $\BMO(\Tr)$ and $\Hone(\Tr)$ that appear in the statement of Theorem \ref{th:theoremC} are the martingale ones adapted to the natural filtration induced by the Gromov metric, and are introduced in Section \ref{sec1}. In particular, we have $\Hone(\Tr)^* = \BMO(\Tr)$ and interpolation with the $L^p(\Tr)$ scale. In our radial setting, the kernel of $\Proj$ satisfies \eqref{eq:sizeCond} and \eqref{eq:HormCond}, and so we can apply Theorem \ref{th:theoremC} to it and we obtain its $L^p$-boundedness as a consequence. We also recover the weak $L^1$-estimates obtained in \cite{dMMV2023,dMMR2023} in the homogeneous case and extend it to the unbounded branching setting. We expect that Theorem \ref{th:theoremC} will be applicable to other integral operators on |non-necessarily radial| trees with unbounded branching behavior. Our nondoubling results do not fit the framework of \cite{Hy2010}, because $\mu$ does not satisfy an appropriate growth condition. Some elements in the proof of Theorem \ref{th:theoremC} are reminiscent of the arguments that naturally appear in nonhomogeneous Calder\'on-Zygmund theory \cite{To2001}, adapted to our specific discrete setting.

Our study of $\Proj$ allows us to prove pointwise estimates for its kernel that resemble the classical Lipschitz smoothness ones for \CZ operators. This and the fact that $\Proj \one_\Tr = \one_\Tr$ suggests that the $\Hone$ and $\BMO$ endpoint estimates are suboptimal. 

\begin{ltheorem} \label{th:theoremD}
The following endpoint results hold for $\Proj$:
\begin{enumerate}
        \item[(a)] $\Proj$ maps $\Hone(\Tr)$ to $\Hone(\Tr)$.
        \item[(b)] $\Proj$ maps $\BMO(\Tr)$ to $\BMO(\Tr)$.
    \end{enumerate}    
\end{ltheorem}

The $\BMO$ estimate for $\Proj$ is new even in the homogeneous setting, although the slightly weaker $L^\infty-\BMO$ one had already been noted explicitly and could be deduced from the known results of Calder\'on-Zygmund theory in homogeneous spaces. The scheme of proof of the $\Hone-\Hone$ estimate seems entirely new. We emphasize once more that our framework is significantly different from all those previously studied in nonhomogeneous harmonic analysis: the measure $\mu$ does not exhibit polynomial (or upper doubling) growth (so \cite{DoVa2023} cannot be followed for the BMO estimate), it is not balanced |in the sense of \cite{LSMP2012}|, and even if it clearly satisfies the isoperimetric property, it is not locally doubling, ruling out the techniques of \cite{CMM2009,CMM2010}.

The rest of the paper is organized as follows: in Section \ref{sec1} we explain our setting, the martingale filtration that we will be working with and some examples of operators that fit our framework. We also collect there some useful kernel estimates for $\Proj$ that we use later in the paper. These are deduced from a novel decomposition of the operator into localizing projections. Section \ref{sec2} contains our sparse domination result for $\Proj$ and the weighted inequalities that descend from it, in terms of the appropriate Bekoll\'e-Bonami-like weight classes. Finally, in Section \ref{sec4} we develop our discrete non-locally doubling \CZ theory and prove Theorems \ref{th:theoremC} and \ref{th:theoremD}.


\section{Trees with unbounded geometry} \label{sec1}

\subsection{Radial trees of (possibly) unbounded degree}

A tree $\Tr$ is a connected graph without loops. We denote the natural graph distance on $\Tr$ by $\graphdist$. Two vertices $x,y \in\Tr$ are  called adjacent if $\graphdist(x,y)=1$. We choose a vertex in $\Tr$ that we call origin and denote by $\orig$. For each $x\in \Tr$, we write
$$
|x|:=\graphdist(x,\orig).
$$
For each $x\in\Tr\setminus \{\orig\}$, we denote
$$
\qf(x)+1 = \#\{y \in \Tr: \; \graphdist(x,y)=1\},
$$
and we declare $\qf(\orig) = \#\{y \in \Tr: \; |y|=1\}$. In that way, $\qf(x)$ denotes the number of children of $x\in\Tr$, i.e. the vertices adjacent to $x$ whose distance from $0$ is $|x|+1$. As we stated in the Introduction, we work with radial trees. This means that we postulate the existence of a function $\qfint:\Z_+ \to \{2,3,4,\ldots\}$ such that
$$
\qf(x) = \qfint(|x|) \quad \mbox{for all} \; x \in \Tr.
$$
Given $\orig \neq x \in \Tr$ with $|x|=\ell$ and $1\leq k\leq \ell$, we denote by $x^{(k)}$ the only vertex $y$ in the geodesic merging $\orig$ and $x$ such that $|y|=\ell-k$. Conversely, the family of vertices $y$ such that $y^{(1)}=x$ is denoted $\sons{x}$. Fix $\alpha >1$. We consider the radially decreasing measure on $\Tr$ defined by 
$$
\mu(\{x\}) := \prod_{k=1}^{|x|} \frac{1}{\qf(x^{(k)})^\alpha}=\prod_{\ell=0}^{|x|-1} \frac{1}{\qfint(\ell)^\alpha}.
$$
Often, we abuse notation and write $\mu(x) = \mu(\{x\})$ for measures $\mu$ on $\Tr$. Without loss of generality, we may assume that $\mu$ has been normalized to be a probability measure. The sector generated by $x\in \Tr$ is the set
$$
\Sec_x = \left\{y \in \Tr: y^{(k)}=x \quad \mbox{for some }k\geq 1 \right\}\cup \{x\}.
$$
Due to the fast decay of $\mu$, we have, and will repeatedly use, the fact that $\mu(\Sec_x) \sim \mu(x)$. The point $x$ is called the vertex of the sector $\Sec_x$. Given a sector $S$, we will sometimes use the notation $x_S$ to denote the vertex of $S$. Given $x,y\in\Tr$, their confluent $x \wedge y$ is the point $z$ that belongs both to the chain $(x^{(k)})_k$ and $(y^{(\ell)})_\ell$ such that $|z|$ is largest. 

\begin{remark}
An alternative construction of a tree can be carried over without distinguishing an origin $\orig$, by hanging $\Tr$ by (a point in) its boundary $\partial\Tr$. $\partial\Tr$ can be constructed as follows: let $\mathcal{G}(\Tr)$ be the space of infinite geodesics |infinite sequences $(x_n)_{n}$ of adjacent points such that $x_{n+2}\neq x_n$ for every $n$|. On $\mathcal{G}(\Tr)$, we define the equivalence relation
$(x_n)_{n}\sim(y_m)_{m}$ if and only if there exists $N\in\Z$ and $n_0\in\N$ such that $x_n=y_{n+N}$ for every $n>n_0$. Then 
$$
\partial\Tr:=\mathcal G(\Tr)/\sim.
$$
Once we fix a hanging point $\omega\in\partial\Tr$, we can define a level structure and then a non-finite, $\sigma$-finite measure adapted to it on $\Tr$. We proceed along the lines of \cite{dMMR2023}, but we need to make some modifications because our tree is no longer homogeneous, and the property of being radial is (can be) conditional on the choice of the boundary point over which $\Tr$ hangs. Indeed, assume $\omega \in \partial\Tr$ has been picked and for each $y\in \Tr$ denote by $(y_n)_n$ the (unique) representative of $\omega$ such that $y_0=y$. We then set $y^{(n)}=y_n$ for all $y$. Fix $x_0\in \Tr$. We proclaim that $x_0$ belongs to the $0$-th horocycle and $x_n$ belongs to the $(-n)$-th horocycle, for $n \geq 1$. All other points $y$ are assigned a horocycle as follows: if $n_1$ and $n_2$ are the smallest numbers such that $y^{(n_1)} = x^{(n_2)}$ then $y$ belongs to the $(n_1-n_2)$-th horocycle. Having fixed $\omega$, $\Tr$ is radial if for each $k\in \mathbb{Z}$, the following holds: for all $y$ in the $k$-th horocycle the set $\{z \in \Tr: z^{(1)}=y\}$ has the same number $q(y)=q_k$ of elements. Then we may define a measure $\mu$ on $\Tr$ as follows: we set $\mu(\{x_0\})=1$, and then we impose the relation
$$
\mu(\{y\})=\frac{\mu(\{y^{(1)}\})}{q(y^{(1)})^\alpha},
$$
which defines $\mu$ over the whole $\Tr$. 

The measure $\mu$ is not finite, but all sectors have finite measure and are very similar to the trees we consider below. In fact, any tree with origin can be seen as a sector in a non-finite measure tree as the ones constructed above. All of our results below can be generalized to radial trees of infinite measure writing $\Tr$ as an increasing limit of sectors. We leave the concrete details and adjustments needed in each estimate to the interested reader. 
\end{remark}

\subsection{Martingale structure. $\Hone$ and $\BMO$} We construct a (one-sided) martingale filtration induced by the Gromov metric in $\Tr$. Each $\sigma$-algebra is generated by a partition $\mathscr{D}_k$, which is trivial for $k=0$ and is constructed as follows for $k\geq 1$:
$$
\D_k = \Big\{ \{x\}: |x| < k\Big\} \bigcup \Big\{ S_y: |y| = k \Big\}.
$$
We use notation inherited from that of the usual dyadic system: we use the letters $Q,R,S$ to denote elements of $\D$, and the only $S\in \D_k$ that contains a given $Q\in \D_{k+s}$ is denoted by $Q^{(s)}$. When $Q=\{x\}$ is a singleton, we sometimes write $x=x_Q$. Clearly, if $\{x\} \in \D_k$ for some $k$, then $\{x\} \in \D_\ell$ for all $\ell > k$. It will be useful for us to always view singletons as points elements in the dyadic generation $\D_\ell$ with the smallest possible value of $\ell$. This means that we will always think that $\{x\}\in \D_k$ if 
$$
\min\left\{\ell: \{x\} \in \D_\ell\right\} = k.
$$
Also, when needed, we will denote
$$
\D_t(Q) = \{R \in \D: R^{(t)} =Q\},
$$
and
$$
\D_{\leq t}(Q) = \bigcup_{k=0}^t \D_k(Q).
$$
The conditional expectations induced by $\D=\{\D_k\}_{k\geq 0}$ are denoted
$$
\cexp_{\D_k} f = \cexp_k f := \sum_{Q \in \D_k} \avg{f}{Q} \one_Q,
$$
while martingale differences are
$$
  \martdiff_k f = \cexp_k f - \cexp_{k-1} f.
$$
In our endpoint estimates we use the spaces $\Hone(\Tr)$ and $\BMO(\Tr)$ defined with respect to the filtration generated by $\D$. The norm in the latter is given by 
$$
\|f\|_{\BMO} = \|\cexp_0 f\|_\infty + \sup_{k\geq 1} \left\| \cexp_k \left| f - \cexp_{k-1}f \right|\right\|_{\infty},
$$
We will use the following useful expression to compute $\BMO$ norms in the sequel:
\begin{equation}\label{eq:equivnorm}
    \|f\|_{\BMO}\sim \sup_{Q\in\D}\frac{1}{\mu(Q)} \sum_{x\in Q} |f(x)-\avg{f}{ Q}|\mu(x)+\sup_{\substack{Q\in\D\\Q\ne\Tr}} |\avg{f}{Q}-\avg{f}{Q^{(1)}}|
    +\left|\sum_{x\in\Tr}f(x)\mu(x)\right|.
\end{equation}
The norm in $\Hone(\Tr)$ is defined via the martingale square function:
$$
\|f\|_{\Hone} = \left\|\left(\sum_{k\geq1}\left|\martdiff_k f \right|^2\right)^{\frac12} \right\|_1.
$$
It will be useful to work with atomic descriptions of $\Hone(\Tr)$. In what follows, an algebraic atom is a function $b\in L^1(\Tr)$ which is supported on some $Q^{(1)}\in \D_k$ for some $k$ |which cannot be a single point| and has vanishing integral. Following \cite{CoPa2016}, an algebraic atom is an atomic block if 
$b=\sum_i \lambda_i a_i$, where $\supp(a_i)\subseteq Q_i$ for some $Q_i\in\D_{k_i}$, $k_i\ge k$ and
\begin{equation*}
\|a_i\|_\infty\le\mu(Q_i)^{-1}\frac{1}{k_i-k+1}.
\end{equation*}
Every $a_i$ is said to be a subatom. For every atomic block we define
\begin{equation*}
|b|^1=\inf_{\substack{b=\sum_i\lambda_i a_i\\a_i \,\rm{ subatoms}}} \sum_{i=1}^{\infty}|\lambda_i|.
\end{equation*}
Then $f\in \Hone(\Tr)$ if and only if $f=\sum_j b_j$, where each $b_j$ is an atomic block, and 
\begin{equation*}
    \|f\|_{\Hone}=\inf_{\substack{f=\sum_j b_j\\b_j \,\text{atomic}\,\text{block}}} \sum_{j=1}^{\infty}|b_j|^1
    =\inf_{\substack{f=\sum_jb_j\\b_j=\sum_i \lambda_{j,i} a_i}} \sum_{j,i=1}^{\infty}|\lambda_{j,i}|.
\end{equation*}
Atomic blocks can be assumed to adopt a simpler form than those, as we show next. An algebraic atom $s$ is called a simple algebraic atom if there exists $a:\Tr \to \R$ such that $\supp(a) \subseteq Q$ and 
$$
s = a - \avg{a}{Q^{(1)}}\one_{Q^{(1)}}.
$$

\begin{lemma}\label{lem:simpleAtoms}
Every function $f\in \Hone(\Tr)$ can be written as
$$
f = \sum_j \lambda_j s_j, 
$$
where each simple atomic block $s_j = a_j - \avg{a_j}{Q_j^{(1)}}\one_{Q_j^{(1)}}$ is such that
$$
\|a_j\|_\infty \leq \mu(Q_j)^{-1}.
$$
Moreover, we have
$$
\|f\|_{\Hone} \sim \inf_{\substack{f=\sum_j\lambda_j s_j\\s_j \,\text{simple} \, \text{atomic}\,\text{block}}} \sum_{j=1}^{\infty}|\lambda_j| =: \|f\|_{\Hone_{\mathrm{sb}}}.
$$
\end{lemma}

\begin{proof}
Clearly, $\|f\|_{\Hone} \lesssim \|f\|_{\Hone_{\mathrm{sb}}}$, for simple atomic blocks are atomic blocks. Indeed, if $b=a - \avg{a}{Q}\one_{Q}$ is a simple atomic block then $a_1 = a$ and $a_2= - \avg{a}{Q}\one_{Q}$ are subatoms satisfying the right size condition. Therefore, the infimum defining $\|\cdot\|_{\Hone}$ ranges over a larger set than that defining $\|\cdot\|_{\Hone_{\mathrm{sb}}}$.

Conversely, let $b$ be an atomic block supported on $Q\in\D_k$ with $|b|^1 = 1$. It is enough to show that $\|b\|_{\Hone_{\mathrm{sb}}} \lesssim 1$. Let $b=\sum_j \lambda_j a_j$ be a decomposition of $b$ into subatoms such that $\sum_j |\lambda_j| \sim 1$. For each $j$, denote the support of $a_j$ by $Q_j \in \D_{k_j}$ and write
$$
a_j = a_j-\avg{a_j}{Q_j^{(1)}}\one_{Q_j^{(1)}} + \sum_{\ell=1}^{k_j-k-1}\left[\avg{a_j}{Q_j^{(\ell)}}\one_{Q_j^{(\ell)}}-\avg{a_j}{Q_j^{(\ell+1)}}\one_{Q_j^{(\ell+1)}} \right] + \avg{a_j}{Q}\one_{Q}.
$$
Since
$$
\sum_j \lambda_j \avg{a_j}{Q} = 0 
$$
because of the mean $0$ of $b$, we can write
\begin{align*}
b =& \sum_j \lambda_j \left[ a_j-\avg{a_j}{Q_j^{(1)}}\one_{Q_j^{(1)}}\right] \\
& + \sum_j \lambda_j \frac{1}{k_j-k+1}\sum_{\ell=1}^{k_j-k-1}(k_j-k+1)\left[\avg{a_j}{Q_j^{(\ell)}}\one_{Q_j^{(\ell)}}-\avg{a_j}{Q_j^{(\ell+1)}}\one_{Q_j^{(\ell+1)}} \right],
\end{align*}
which we claim is a decomposition of $b$ into simple atomic blocks. Indeed, for each $j$ and $1\leq \ell \leq k_j-k+1$,
$$
\left\|(k_j-k+1)\avg{a_j}{Q_j^{(\ell)}}\one_{Q_j^{(\ell)}}\right\|_\infty \leq \frac{(k_j-k+1)\mu(Q_j)}{\mu(Q_j^{(\ell)})}\|a_j\|_\infty \leq \frac{1}{\mu(Q_j^{(\ell)})},
$$
as required. An entirely similar estimate holds for the terms of the form $a_j-\avg{a_j}{Q_j^{(1)}}\one_{Q_j^{(1)}}$. Therefore, 
$$
\|b\|_{\Hone_{\mathrm{sb}}} \lesssim \sum_j |\lambda_j|\left(1+\sum_{\ell=1}^{k_j-k-1}\frac{1}{k_j-k+1}\right) \lesssim \sum_j |\lambda_j| \sim 1,
$$
which concludes the proof.
\end{proof}

\begin{remark}
    The proof of Lemma \ref{lem:simpleAtoms} can clearly be done in any probability space, and not only in one whose associated filtration is atomic. We do not need that level of generality here and so we omit the details. 
\end{remark}

\begin{remark}
    Above, we have dealt with atoms satisfying an $L^\infty$ size condition. Equivalent norms and notions of both atomic blocks and simple atomic blocks can be constructed for each $1<p<\infty$ replacing said size condition by
    $$
    \|a_j\|_p \leq \frac{1}{\mu(Q_j)^{\frac{1}{p'}}}.
    $$
    Lemma \ref{lem:simpleAtoms} holds as well for $L^p$-atomic blocks. We use the case $p=2$ below in Section \ref{sec4}.
\end{remark}

\subsection{Operators acting on $\Tr$: Bergman spaces and projection-like operators} The combinatorial Laplacian on $\Tr$ is the operator defined by
$$
\Delta f (x) = f(x) - \frac{1}{q(x)+1}\sum_{y:\graphdist(x,y)=1} f(y).
$$
A function $f$ is called harmonic at a point $x$ if $\Delta f(x) = 0$, and is called harmonic if it is harmonic at all $x\in\Tr$. The (real) Bergman space is
$$
\Berg^2(\Tr) = \left\{f\in L^2(\Tr): \; f \mbox{ is harmonic} \right\}.
$$
This is a real-variable version of the classical Bergman space. $\Berg^2(\Tr)$ is a reproducing kernel Hilbert space for which an explicit orthonormal basis can be exhibited. 
To do so, we start defining a family of auxiliary functions. For each $k\geq0$, we denote by $\RH_k$ the only function on $\Tr$ such that $\RH_k(x)=0$ if $|x|<k$, $\RH_k(x)=1$ if $|x|=k$, that is harmonic at all points $x$ such that $|x|>k$, and is radial: $\RH_k(x)=\RH_k(y)$ if $|x|=|y|$. The function $\RH_k$ encodes the decay properties of the kernels studied below, in terms of the following quantity:
$$
\nu_j^k := \prod_{\ell=j}^{k} \frac{1}{\qfint(\ell)},\quad k\ge j
$$
with the convention that $\nu_j^k=1$ if $j > k$. Many times, the trivial estimate $\nu_j^k \le~2^{-(k-j+1)}$ will be enough for our purposes, but the unboundedness of $q$ will often force us to use the larger decay captured by $\nu_j^k$. We will sometimes use the multiplicative identity $ \nu_j^k= \nu_j^m \cdot \nu_{m+1}^k$ for $j\leq m\leq k$. The following explicit formula, which can be checked by induction, allows one to compute the values of $\RH_k$:
\begin{equation}\label{eq:explicitRH}
\RH_k(x) = \sum_{\ell=1}^{|x|-k+1} \nu_{k}^{|x|-\ell},\quad |x|\geq k.    
\end{equation}
Formula \eqref{eq:explicitRH} immediately implies that $1\leq\RH_k(x) \leq 2$ for all $x\in\supp(\RH_k)$, a fact that we will use repeatedly. The following more precise estimate will be useful in the sequel: 
\begin{lemma}\label{lem: RH}
Let $k\geq 0$. If $x,y$ are such that $|x|\geq|y|=m \geq k$, then 
\[ 
\RH_k(x)-\RH_k(y)\lesssim \nu_k^{m}.
\]
\end{lemma}
\begin{proof}
Indeed,
\begin{align*}
\RH_k(x)-\RH_k(y) & = \sum_{\ell=1}^{|x|-k+1} \nu_{k}^{|x|-\ell}-\sum_{\ell=1}^{|y|-k+1} \nu_{k}^{|y|-\ell} \\
& = \sum_{\ell=|x|-|y|+1}^{|x|-k+1} \nu_{k}^{|x|-\ell} - \sum_{\ell=1}^{|y|-k+1} \nu_{k}^{|y|-\ell} + \sum_{\ell=1}^{|x|-|y|} \nu_{k}^{|x|-\ell} \\
& = \nu_k^{m} \sum_{\ell=1}^{|x|-|y|} \nu_{m+1}^{|x|-\ell} \lesssim \nu_k^m,\\
\end{align*}
using $\nu_{m+1}^{|x|-\ell} \lesssim 2^{-|x|+\ell+m}$.
\end{proof}

Given a sector $Q\in\D_{k-1}$ and a function $e:Q \cap\{|x|=k\}\to\R$ with mean zero, the function
\begin{equation*}
    E_Q(e)(x)=\begin{dcases*}
        \RH_k(x)e(y)&$x\in S_y$ for some $y\in\D_1(x_Q)$,\\
        0&otherwise,
    \end{dcases*}
\end{equation*}
is harmonic on $\Tr$ and we call it the radial harmonic extension of $e$. For each $x\in\Tr$, denote $Q=\Sec_x$ and pick an orthonormal basis $\{e^\ell\}_{\ell=1}^{q(x)-1}$ of the hyperplane
$$
W_Q = \left\{v=(v_j) \in \R^{q(x)}: \sum_j v_j=0 \right\}.
$$
Let $\sons{x}=\{y_j:j=1,\dots,q(x)\}$. For each $\ell$, define the function $e_Q^\ell$ as the one supported in the successors of $x$ that satisfies $e_Q^\ell(y_j)=e_j^\ell$. Finally, for each $\ell$ and $Q$ set
$$
\h_Q^\ell(z) = \frac{E_Q(e_Q^\ell)(z)}{\|E_Q(e_Q^\ell)\|_2}, \quad z \in \Tr.
$$
When $Q\in\D$ is a singleton, by convention we define $h_Q\equiv0$. The next result uses a key property of the functions $\h_Q^\ell$: they are not only orthogonal to constants, but moreover they are orthogonal to all radial functions (and they themselves behave like radial functions in almost all sectors).

\begin{prop}\label{prop:constructionofBasis}
The family 
$$
\mathcal{U}=\left\{\h_Q^\ell:\;\;Q\in\D,1 \leq\ell \leq q(x_Q)-1 \right\} \cup \left\{\one_{\Tr}\right\}
$$ is an orthonormal basis of $\Berg^2(\Tr)$.  
\end{prop}

\begin{proof}
The functions under consideration are $L^2$-normalized, so we have to check orthogonality and completeness. First, if $Q\in\D_{k-1}$ and $\ell\neq\ell'$ then
\begin{align*}
    \langle \h_Q^\ell, \h_Q^{\ell'} \rangle & = \sum_{x\in \sons{x_Q}} \sum_{y\in\Sec_x} e_Q^\ell(x)e_Q^{\ell'}(x) \RH_k(y)^2 \mu(y) \\
    & = \sum_{x\in \sons{x_Q}}e_Q^\ell(x)e_Q^{\ell'}(x)\sum_{y\in\Sec_x}\RH_k(y)^2 \mu(y) \\
    & = C_{k,Q} \sum_{x\in \sons{x_Q}}e_Q^\ell(x)e_Q^{\ell'}(x)=C_{k,Q} \sum_{j=1}^{q(x_Q)} e_j^\ell e_j^{\ell'}=0,
\end{align*}
by the orthogonality of the vectors that generate $\h_Q^\ell$ and $\h_Q^{\ell'}$. Second, since $\supp(\h_Q^\ell) \subseteq Q$ for all $Q$, the following |similar| computation for $\D_{r-1} \ni Q \subsetneq R \in \D_{k-1}$ deals with all the remaining cases:
\begin{align*}
\langle \h_Q^\ell, \h_R^{\ell'} \rangle & = \sum_{x\in \sons{x_Q}} \sum_{y\in\Sec_x} e_Q^\ell(x)e_R^{\ell'}(x_0) \RH_r(y)\RH_k(y) \mu(y) \\ 
& = e_R^{\ell'}(x_0) \sum_{x\in \sons{x_Q}} e_Q^\ell(x) \sum_{y\in\Sec_x}\RH_r(y)\RH_k(y) \mu(y)\\
& = C_{k,r,Q} e_R^{\ell'}(x_0) \sum_{x\in \sons{x_Q}} e_Q^\ell(x) = 0.
\end{align*}
Above, the point $x_0\in\sons{x_R}$ is the only one such that $Q\subseteq S_{x_0}$. Finally, all functions $h_Q^\ell$ have mean zero so they are orthogonal to constants. 

It remains to show that $\mathcal{U}$ generates the whole space. Suppose $f\in\Berg^2(\Tr)$ is orthogonal to every element of $\mathcal{U}$. In particular, $f$ is orthogonal to constants and it can then be checked that $f(\orig)=0$. We assume by induction that $f(x) = 0$ if $|x|<n$ and we prove that $f(y)=0$ for all $y$ with $|y|=n$. Fix $x$ with $|x|=n-1$ and denote $Q=S_x$. Since $f(x)=f(x^{(1)})=0$, we have
$$
\sum_{y\in\sons{x}} f(y)=0.
$$
But then, for all $\ell$
\begin{align*}
0=\langle \h_Q^\ell, f \rangle & = \sum_{y\in \sons{x}} \sum_{z\in\Sec_y} e_Q^\ell(y)f(z) \RH_n(z) \mu(z) \\ 
& = \sum_{y\in \sons{x}} e_Q^\ell(y)\sum_{z\in\Sec_y} f(z) \RH_n(z) \mu(z) \\ 
& = \sum_{y\in \sons{x}} e_Q^\ell(y)\sum_{j=1}^\infty \sum_{\substack{z\in\Sec_y\\|z|=|x|+j}} f(z) \RH_n(z) \mu(z) \\ 
& = \sum_{y\in \sons{x}} e_Q^\ell(y)\sum_{j=1}^\infty C_{j,n} f(y) \\
& = C_n\sum_{y\in \sons{x}} e_Q^\ell(y) f(y),
\end{align*}
using the radiality of $\mu$ and $\RH_n$. This implies that $f(y)=0$ for all $y\in\sons{x}$.
\end{proof}
Since pointwise evaluation is a bounded functional, $\Berg^2(\Tr)$ is a closed subspace of $L^2(\Tr)$, as in the continuous case. The Bergman projector $\Proj$ is the orthogonal projection $L^2(\Tr)\mapsto \Berg^2(\Tr)$. Because of Proposition \ref{prop:constructionofBasis}, $\Proj$ has the explicit expression
$$
\Proj f(x) =\avg{f}{\Tr}\one_\Tr(x) + \sum_{Q \in \D} \sum_{\ell=1}^{q(x_Q)-1}\langle f,h_Q^\ell\rangle h_Q^\ell(x).
$$
$\Proj$ is an integral operator with a kernel that can be written explicitly. But as we said in the Introduction, it is better to work with it adopting a different point of view. To that end, the following definition does not only ease our notation in the sequel, but it also encodes the right way of manipulating the pieces of the kernel of $\Proj$. For each $Q \in \D$ which is not a singleton, we set
$$
\Diff_Q f(x) := \sum_{\ell=1}^{\qf(x_Q)-1} \langle f, \h_Q^\ell\rangle \h_Q^\ell(x).
$$
When $Q$ is a singleton we set $\Diff_Q f = 0$. Clearly, for any $f$ we have $\supp(\Diff_Q f) \subseteq Q \setminus \{x_Q\}$ and $\Diff_Q f$ has vanishing mean. We will sometimes work with truncated versions of $\Proj$, which we denote
\[
\Proj^Rf:= \sum_{\substack{T\in\D\\T\subseteq R}}\Diff_Tf,
\]
in the case of the localized truncation inside $R\in\D$, and for a disjoint family $\mathcal{F}\subseteq \D$
\[
\Proj_{\mathcal{F}}f:=\avg{f}{\Tr}\one_\Tr + \sum_{\substack{T\in\D\\T\not\subseteq R \; \mathrm{ if }\; R\in\mathcal{F}}}\Diff_Tf.
\]
\begin{remark}\label{rem:remBergShifts}
Following the classical definition of Haar shifts, given a pair $(r,s) \in \N_+^2$ we define
$$
\Proj_\gamma^{r,s}f(x) = \sum_{Q \in \D} \sum_{R \in \D_r(Q)} \sum_{S \in \D_s(Q)} \sum_{k=1}^{\qf(v_R)-1} \sum_{\ell=1}^{\qf(v_S)-1} \gamma^Q_{R,S,k,\ell}\langle f, \h_R^k\rangle \h_S^\ell(x).
$$
with $\gamma \in \ell^\infty$. We call $\Proj_\gamma^{r,s}$ a Bergman shift of complexity $(r,s)$. The Bergman projector corresponds to the case $\Proj_{\gamma_0}^{0,0}$, where 
$$
\gamma_0 = \gamma^Q_{Q,Q,k,\ell}=\delta_{k,\ell}.
$$
It should be observed that, unless $r=s=0$, $\Proj_\gamma^{r,s}$ cannot be expressed as a sum involving only the operators $\Diff_Q$. This is one way to see that, when $\qf$ is unbounded, one should not expect to prove $L^p$-boundedness results for $\Proj_\gamma^{r,s}$ if $p\neq 2$. The heuristic reason comes from \cite{LSMP2012}, where the family of measures such that a Haar shift and its adjoint are $L^p$-bounded is characterized. The measure $\mu$ does not belong to that class, and in an analogous way one should expect $L^p$-boundedness to fail. It is possible to prove positive boundedness results for $\Proj_\gamma^{r,s}$ when $q$ is bounded.      
\end{remark}

\subsection{Kernel estimates for $\Proj$} We next collect two key estimates involving the operators $\Diff_R$ that will be useful in the following sections. By definition, $\Diff_R$ is an integral operator with kernel $\K_R$ given by
\[
\K_R(x,y)=\sum_{j=1}^{q(x_R)-1} h_R^j(x) h_R^j(y). 
\]
On the other hand, the space $W_R$ introduced above is a reproducing kernel Hilbert space. Denoting by $k_R(w,w')= \sum_{j=1}^{q(x_R)-1}e^j_R(w)e^j_R(w')$ its reproducing kernel, for $R\in\D_{k-1}$ we have
\begin{align}\label{eq: K&k}
\K_R(x,y)&=\sum_{j=1}^{q(x_R)-1} h_R^j(x)h_R^j(y)=\frac{q(x_R)\RH_k(x)\RH_k(y)}{\|\RH_k\one_{R}\|_2^2}\sum_{j=1}^{q(x_R)-1} e^j_R(x^{(|x|-|x_R|-1)})e_R^j(y^{(|y|-|x_R|-1)})\nonumber\\
&=\frac{q(x_R)\RH_k(x)\RH_k(y)}{\|\RH_k\one_{R}\|_2^2}k_R(x^{(|x|-|x_R|-1)},y^{(|y|-|x_R|-1)}).
\end{align}
Finally, a direct computation shows that $k_R$ is uniquely determined by 
\begin{equation*}
k_R(w,w')=\begin{dcases*}-\frac 1{q(x_R)},&\text{if }$w\neq w'$,\\
	1-\frac 1{q(x_R)},&\text{if }$w= w'$.
	\end{dcases*}\qquad (w,w')\in W_R \times W_R.
\end{equation*}

\begin{lemma}\label{lem: KrestF}
If $Q\in\D_{k}$ is not a single point, the following two estimates hold:
\begin{enumerate}
        \item For all $\ell\geq1$ and every $z\in Q^{(\ell)}$ we have 
        \begin{equation*}
            \sup_{x,y\in Q}|\K_{Q^{(\ell)}}(z,x)-\K_{Q^{(\ell)}}(z,y)|\lesssim\begin{dcases*}
            \frac{1}{\mu(Q^{(\ell-1)})}\nu_{k-\ell+1}^{k},&if $z\in Q^{(\ell-1)}$,\\
               \frac{1}{\mu(Q^{(\ell-1)})}\nu_{k-\ell}^{k},&otherwise.
            \end{dcases*}
        \end{equation*}
        \item For all $x,y \in Q^{(1)}\setminus \{x_Q^{(1)}\}$,
        \begin{equation*}
            |\K_{Q^{(1)}}(x,y)|\lesssim\begin{dcases*}
                \frac{1}{\mu(Q)},&if $x\wedge y\neq\Or Q^{(1)}$,\\
                \frac{1}{q(\Or Q^{(1)})\mu(Q)},&otherwise.
            \end{dcases*}
        \end{equation*}
    \end{enumerate}
\end{lemma}
\begin{proof}
	\begin{enumerate}
		\item By the triangle inequality, it is enough to prove the assertion for $y=x_Q$. Since $1\leq\RH_{k-\ell+1}<2$, using \eqref{eq: K&k} and Lemma~\ref{lem: RH} yields
		\begin{align*}
			\K_{Q^{(\ell)}}(z,x)&-\K_{Q^{(\ell)}}(z,x_Q) \\
			& =\frac{q({x_Q^{(\ell)}})\RH_{k-\ell+1}(z)(\RH_{k-\ell+1}(x)-\RH_{k-\ell+1}(x_Q))}{\|\RH_{k-\ell+1}\one_{\anc Q{\ell}}\|_2^2}k_{Q^{(\ell)}}(z^{(|z|-k+\ell-1)},x_Q^{(\ell-1)})\\
			&\sim \nu_{k-\ell+1}^{k}\frac {1}{\mu(Q^{(\ell-1)})}k_{Q^{(\ell)}}(z^{(|z|-k+\ell-1)},x_Q^{(\ell-1)})\\
			&\sim \begin{dcases*}
			    \frac{1}{\mu(Q^{(\ell-1)})}\nu_{k-\ell+1}^{k},&if $z\in Q^{(\ell-1)}$,\\
            -\frac{1}{q(x_Q^{(\ell)})\mu(Q^{(\ell-1)})}\nu_{k-\ell+1}^{k}, &otherwise.
			\end{dcases*}
		\end{align*}
 \item Again by \eqref{eq: K&k},
            \begin{align*}
                \K_{Q^{(1)}}(x,y) & =\frac{q(\Or Q^{(1)})\RH_{k}(x)\RH_{k}(y)}{\|\RH_{k}\one_{Q^{(1)}}\|_2^2}k_{Q^{(1)}}(x^{(|x|-k)},y^{(|y|-k)})\\
                & \sim \begin{dcases*}
                \frac{1}{\mu(Q)},&if $x\wedge y\neq x_Q^{(1)}$,\\
                -\frac{1}{q(\Or Q^{(1)})\mu(Q)},&otherwise.
            \end{dcases*}            
            \end{align*}
\end{enumerate}
\end{proof}
We next prove an estimate for truncations of $\Proj$. Given a family $\mathcal{F}$ of disjoint sets in $\D$ contained in a sector $Q_0$, we denote the kernel of the truncation $\PrestF$ by $\KrestF$, which clearly can be written as
\begin{equation}\label{eq: KrestF}
    \KrestF= \sum_{\substack{Q\in\D,\, Q\subseteq Q_0 \\Q\not\subseteq R\in\mathcal F}}\K_Q.
\end{equation}
In the particular case when $Q_0=\Tr$, we slightly abuse notation and exclude $\K_{\Tr}$ from the sum |and will deal with it separately below|. 

\begin{prop}\label{prop:Krest}
    Let $Q\in\D_k$ be such that $Q\subseteq Q_0$ and $Q\not\subseteq R$ for any $R\in\mathcal{F}$. For every $z\in Q^{(\ell)}\setminus Q^{(\ell-1)}\subseteq Q_0$, we have 
    \begin{equation*}
        \sup_{x,y\in Q}|\KrestF(z,x)-\KrestF(z,y)|\lesssim\begin{dcases*}
            \frac{1}{\mu(Q^{(\ell)})}\nu_{k-\ell}^{k},&if $z= x_Q^{(\ell)}$,\\
            \frac{1}{\mu(Q^{(\ell-1)})}\nu_{k-\ell}^{k},&if $z\neq x_Q^{(\ell)}$.
        \end{dcases*}
    \end{equation*}
\end{prop}
\begin{proof}
We have 
    \[|\KrestF(z,x)-\KrestF(z,y)|\leq \sum_{j=0}^{k-\ell}|\K_{Q^{(\ell+j)}}(z,x)-\K_{Q^{(\ell+j)}}(z,y)|,\]
    where the first term of the sum vanishes whenever $z=x_Q^{(\ell)}$. By part $(1)$ of Lemma~\ref{lem: KrestF}, one has
    \begin{equation*}
        |\KrestF(z,x)-\KrestF(z,y)|\lesssim
        \begin{dcases*}
            \sum_{j=1}^{k-\ell}\frac{1}{\mu(Q^{(\ell+j-1)})}\nu_{k-\ell-j+1}^{k},&if $z= x_Q^{(\ell)}$\\
            \frac 1{\mu(Q^{(\ell-1)})} \nu_{k-\ell}^{k}+\sum_{j=1}^{k-\ell}\frac{1}{\mu(Q^{(\ell+j-1)})}\nu_{k-\ell-j+1}^{k},
            &if $z\neq x_Q^{(\ell)}$.
        \end{dcases*}
    \end{equation*}
The statement follows from the fact that both sums are geometric and hence controlled by their first term, and the one appearing in the term below is in turn controlled by the one on its left.
\end{proof}

The following notation will sometimes be used in the sequel: if $Q\in\D$, $\cl{Q}$ is the smallest sector that contains it |it is only different from $Q$ if the latter is a singleton|. Our last estimate is particularly useful when controlling local terms.

\begin{prop}\label{prop: estim D}
Let $b$ be a simple algebraic atom associated to $Q\in\D_k$. For every $1\leq \ell\leq|x_Q|$,
        \begin{subnumcases}{\Diff_{\anc {\cl Q}{\ell}}(b)(x)\lesssim
        }
            \frac{ \nu_{k-\ell+1}^{k} }{\mu(Q^{(\ell-1)})}\|b\|_1,
            &
            $\text{if }x\in\cl{Q}^{(\ell-1)}$,\label{eq: estim in}\\
            \frac{\nu_{k-\ell}^{k}}{\mu(Q^{(\ell-1)})}\|b\|_1,
            &$\text{if }x\in\cl{Q}^{(\ell)}\setminus\anc{\cl{Q}}{\ell-1}.$\label{eq: estim out}
        \end{subnumcases}
    Furthermore, if $Q$ is a single point, then $\Diff_{\cl Q}b\equiv 0$.
\end{prop}
\begin{proof}
The case $\ell> 1$ follows from the fact that $b$ has null mean and Lemma~\ref{lem: KrestF} applied to $\anc{\cl Q}1$. Indeed, 
\begin{align*}
    |\Diff_{\anc{\cl{Q}}{\ell}}b(x)|&=\left|\sum_{y\in\anc{Q}{1}}\Bigl(\K_{\anc{\cl Q}\ell}\bigl(x,y\bigr)-\K_{\anc{\cl Q}\ell}\bigl(x,\Or{\anc{\cl Q}{1}}\bigr)\Bigr)b(y)\mu(y)\right|\\
    & \leq \sup_{y\in Q^{(1)}} \left|\K_{\anc{\cl Q}\ell}\bigl(x,y\bigr)-\K_{\anc{\cl Q}\ell}\bigl(x,\Or{\anc{\cl Q}{1}}\bigr)\right| \|b\|_1 \\
    &
    \leq\begin{dcases*}
        \frac{\nu_{k-\ell+1}^{k}}{\mu(Q^{(\ell-1)})}\|b\|_1,
            &
            $\text{if }x\in\cl{Q}^{(\ell-1)}$,\label{eq: estim in}\\
            \frac{\nu_{k-\ell}^{k}}{\mu(Q^{(\ell-1)})}\|b\|_1,
            & $\text{if }x\in\cl{Q}^{(\ell)}\setminus\anc{\cl{Q}}{\ell-1},$
    \end{dcases*}
\end{align*}
regardless of whether $Q=\cl{Q}$ or not. For $\ell=1$, if $Q$ is a single point one can argue similarly as above. Otherwise, denote by $R$ the element so that $x\in R\in \D_1(Q^{(1)})$. By $(2)$ of Lemma~\ref{lem: KrestF}, one has
\begin{align*}
|\Diff_{\anc Q1}b(x)|& \leq \sum_{y\in\anc Q1\setminus\{x_Q^{(1)}\}}|\K_{\anc Q1}(x,y)||b(y)|\mu(y) \\ 
& \lesssim \frac 1{\mu(Q)}\left[\|b\one_R\|_1+\frac{1}{q(x_Q^{(1)})}\|b\one_{\anc Q1\setminus(R\cup\{x_Q^{(1)}\})}\|_1\right].
\end{align*}
If $R=Q$, then clearly $|\Diff_{\anc Q1}b(x)| \lesssim \|b\|_1\mu(Q)^{-1}$. Otherwise, since $b = a - \avg{a}{Q^{(1)}}\one_{Q^{(1)}}$ we have that
    \[|\Diff_{\anc Q1}b(x)|\lesssim \frac 1{\mu(Q)}\left[\mu(R)|\avg{a}{\anc Q1}|+\frac{1}{q(x_Q^{(1)})}\|b\|_1\right]\lesssim \frac{\|b\|_1}{q(x_Q^{(1)})\mu(Q)}.\]
Finally, the fact that $\Diff_{\cl Q}b\equiv 0$ when $Q=\{v\}$ follows from
    \begin{align}\label{eq: Diff vanish k=1}
        \Diff_{\cl Q}b(x)&=-\frac{a(v)}{\mu(\cl Q)}\sum_{y\in\cl Q\setminus\{v\}}\K_{\cl Q}(x,y)\mu(y)\nonumber\\
        &=-\frac{a(v)}{\mu(\cl Q)}\sum_{z\in\sons{v}}k_{\cl{Q}}(x^{|x|-|v|-1},z)\sum_{y\in\Sec_z}\frac{\RH_{|v|+1}(x)\RH_{|v|+1}(y)}{\|\RH_{|v|+1}\one_{\cl Q}\|_2^2}\mu(y)=0,
    \end{align} 
    by the definition of $k_{\cl{Q}}$ and the radiality of $\RH_{|v|+1}$.
\end{proof}

\begin{remark}
    The key to obtaining the estimates above is to exploit the joint cancellation enjoyed by the family $\{\h_Q^\ell\}_\ell$ for each fixed $Q$. We do that by studying the operators $\Diff_Q$ as a whole instead of considering the individual products $\langle f, \h_Q^\ell \rangle \h_Q^\ell$, of which there can be arbitrarily many within each cube $Q$, and for which size estimates are not precise enough. This is a key difference with respect to the usual point of view that one adopts in dyadic harmonic analysis, in which the action of each function in the Haar basis is dealt with separately. In a sense, one can interpret our point of view as a way to dealing with the fact that $\Tr$ may be infinite-dimensional.
\end{remark}

\section{Sparse forms associated with $\Proj$} \label{sec2}

This section is devoted to the proof of Theorems \ref{th:theoremA} and \ref{th:theoremB}. We start defining the form $\Extras$ that appears in the statement of Theorem \ref{th:theoremA}. If $\Ss \subseteq \D$ is a sparse family we put
\begin{equation}\label{eq:sparsenew}
\Extras(f_1,f_2):=\sum_{Q=\cl{Q}\in \Ss} \avg{f_1}{Q} \Bigg(\frac{1}{q(x_Q^{(1)})}\sum_{\substack{R=\cl{R}\in\Ss \\ R^{(1)}=Q^{(1)}}}\avg{f_2}{R}\Bigg) \mu(Q).
\end{equation}

We need to use the following variant of the Calder\'on-Zygmund decomposition for general measures, both in the next Subsection and in Section \ref{sec4}.

\begin{lemma} \label{lem:CZnonhomogeneous}
Let $R\in\D$, $0\leq f\in L^1(R)$ and $\lambda> \mu(R)^{-1}\|f\|_{L^1(R)}$. Let $\{Q_j\}_j \subseteq \D$ be a disjoint family such that 
$$
\left\{\M_\D f > \lambda\right\} \subseteq \bigcup_j Q_j,
$$ 
and $Q_j^{(1)}\not\subseteq \bigcup_{\ell}Q_{\ell}$ for any $j$. Then, if we define $b=\sum_j b_{Q_j}$, where
$$
b_{Q_j} := f \one_{Q_j} - \langle f \one_{Q_j}\rangle_{Q^{(1)}_j} \one_{Q^{(1)}_j},
$$
and $f=g+b$, the following properties are satisfied:
\begin{enumerate}
    \item Localization:
    \[
    \supp(b_{Q_j})\subseteq Q^{(1)}_j,\qquad \sum_{x\in Q^{(1)}_j}b_{Q_j}(x)\mu(x)=0.
    \]
    \item $L^1$-bounds:
    $$
    \|g\|_1 \lesssim \|f\|_1, \quad \sum_j \|b_{Q_j}\|_1 \lesssim \|f\|_1.
    $$
    \item Higher integrability for $g$:
    $$
    \|g\|_2^2 \lesssim \lambda \|f\|_1, \quad \|g\|_{\BMO} \lesssim \lambda.
    $$
\end{enumerate}
\end{lemma}

The statement above is different from the usual one in that the \CZ cubes can be taken to be as large as desired, so long as they are still disjoint (and so smallness of the measure of the exceptional set is not claimed). We take advantage of this additional feature below in the proof of Theorem \ref{th:theoremA}, as is usually done in sparse domination proofs. The proof of Lemma \ref{lem:CZnonhomogeneous} is a straightforward adaptation from \cite{CoPa2019,CCP2022}, and is hence omitted. The key $\BMO$ estimate is proven in \cite{CPW2023}.

\subsection{Sparse domination} \label{sec21} In the proof of Theorem \ref{th:theoremA} we will need to use the uniform $L^2$-boundedness of truncations of $\Proj$, that we record below:
\begin{equation}\label{eq:L2unifBound}
    \sup_{Q_0\in\D,\;\mathcal{F}\subseteq \D} \left\|\Proj_{\mathcal{F}}^{Q_0} \right\|_{L^2(\Tr)\to L^2(\Tr)} < \infty.
\end{equation}
The proof of \eqref{eq:L2unifBound} follows immediately from the orthogonality and normalization of the operators $\Diff_Q$ and we omit it. 

\begin{proof}[Proof of Theorem \ref{th:theoremA}]
Since we will have $\Tr \in \Ss$ for all pairs $(f_1,f_2)$, it is enough to prove sparse domination for the operator $\Proj-\avg{f}{\Tr}\one_\Tr$. Let $f_1,f_2\in L^1(\Tr)$ be nonnegative, and set $Q_0=\Tr$. We have
$$
\langle \Proj f_1, f_2\rangle = \avg{f_1}{Q_0}\|f_2\|_1 + \left\langle \sum_{Q\in\D}\Diff_Q f_1, f_2\right\rangle = \avg{f_1}{Q_0}\avg{f_2}{Q_0}\mu(Q_0) + \langle \Proj^{Q_0}f_1,f_2 \rangle. 
$$
Since $Q_0 \in \Ss$, the first term in the right hand side above is acceptable, and we focus on estimating the second for the remainder of the proof. We denote
\[
\Omega^0_i=\{\M_\D f_i>4\|\M_\D\|_{L^1(\Tr) \to L^{1,\infty}(\Tr)}\langle f_i\rangle_{Q_0} \},\qquad i\in\{1,2\},
\]
and $\Omega^0=\Omega^0_1\cup\Omega^0_2$. We cover $\Omega^0$ with a disjoint family $\mathcal F^1\subseteq \D$, maximal with respect to inclusion. For $i=1,2$, we apply the \CZ decomposition from Lemma~\ref{lem:CZnonhomogeneous} to $f_i$ at height  $\lambda_i=4\|\M_\D\|_{L^1(\Tr) \to L^{1,\infty}(\Tr)}\avg{f_i}{Q_0}$.
We next estimate the non localizing part of $\Proj$, namely the truncation
$$
\Prest f=\Proj^{Q_0} f-\sum_{Q\in\mathcal F^1}\Proj^Qf.
$$
We claim
\begin{align}
|\langle \Prest f_1,f_2\rangle|\lesssim & \; 
        \avg{f_1}{Q_0} \avg{f_2}{Q_0} \mu(Q_0) +  \sum_{Q\in\mathcal{F}^1} \avg{f_1}{Q} \avg{f_2}{Q} \mu(Q) \notag\\
        & + \sum_{Q=\cl{Q}\in\mathcal{F}^1} \avg{f_1}{Q}  \bigg(\frac{1}{q(\Or Q^{(1)})} \sum_{\substack{R=\cl{R}\in\mathcal{F}^1 \\ R^{(1)}=Q^{(1)}}} \avg{f_2}{R}\bigg)\mu(Q). \label{eq: claim first step}
\end{align}
We split according to the \CZ pieces:
\begin{align*}
|\langle \Prest f_1,f_2\rangle|&\leq|\langle \Prest g_1,g_2\rangle|+|\langle \Prest b_1,g_2\rangle|+|\langle \Prest g_1,b_2\rangle|+|\langle \Prest b_1,b_2\rangle|\\
&=:\text{I}+\text{II}+\text{III}+\text{IV},
\end{align*}
and we analyze separately the four cases. For term I, Cauchy-Schwarz's inequality and \eqref{eq:L2unifBound} yield
\begin{align*}
\text{I}\leq \|\Prest g_1\|_2\|g_2\|_2\lesssim\|g_1\|_2\|g_2\|_2&\lesssim (\langle f_1\rangle_{Q_0}\|f_1\|_1\,\langle f_2\rangle_{Q_0}\|f_2\|_1)^{1/2}=\avg{f_1}{Q_0}\avg{f_2}{Q_0}\mu(Q_0).
\end{align*}
To analyze the terms involving the bad parts in the \CZ decomposition, we write
\begin{align}\label{eq: Prestbad}
        \Prest(b_{1,Q})&=\sum_{\substack{D=\cl{D}\in\D\\ D\cap\anc{Q}1\neq\emptyset}}\Diff_D b_{1,Q}-\sum_{\substack{R=\cl{R}\in\mathcal{F}^1\\ R\cap\anc{Q}1\neq\emptyset}}\sum_{\substack{T=\cl{T}\in\D\\ T\subseteq R\\ T\cap\anc{Q}1\neq\emptyset}}\Diff_T b_{1,Q} \nonumber\\
        &=\sum_{k=1}^{|x_Q|} \Diff_{\anc{Q}k}b_{1,Q}+\sum_{\substack{T=\cl{T}\in\D\\ T\subseteq \anc{Q}1\setminus\Omega^0}}\Diff_T b_{1,Q} .
\end{align}
For every $T\subseteq \anc{Q}1\setminus\Omega^0$, we have $b_{1,Q}|_T\equiv \avg{f\one_Q}{Q^{(1)}}$, and so the last sum in~\eqref{eq: Prestbad} vanishes. 
Using that $\Diff_{\cl Q}b_{1,Q}\equiv 0$ when $Q$ is a singleton, the above formula reduces to
    \begin{equation}\label{eq: final Prest}
        \Prest(b_{1,Q})=\sum_{k=1}^{|x_Q|}\Diff_{\anc{\cl Q}{k}}b_{1,Q}.
    \end{equation}
Since $\Prest$ is self-adjoint, it is enough to study II and the estimate for III will follow. For every $1\leq k\leq|x_Q|$,
we put $g_2^{(k)}=g_2-\avg{g_2}{\anc{\cl Q}{k}}$. By~\eqref{eq:equivnorm}, we have $\avg{g_2^{(k)}}{R}\lesssim\|g_2\|_{\BMO}$ for every $R\in\D^1(\anc{\cl{Q}}{k})$. Then, by the vanishing mean of $\Diff_{\anc{\cl Q}{k}}(b_{1,Q})$ and Proposition~\ref{prop: estim D} we have
    \begin{align*}
    \langle \Diff_{\anc{\cl Q}{k}}&(b_{1,Q}),g_2\rangle
    =\langle \Diff_{\anc{\cl Q}{k}}(b_{1,Q}),g_2-\avg{g_2}{\anc{\cl Q}{k}}\rangle\\
    &=\|\Diff_{\anc{\cl Q}{k}}(b_{1,Q})\one_{\anc{\cl Q}{k-1}}\|_\infty\|g_2^{(k)}\one_{\anc{\cl Q}{k-1}}\|_1+\sum_{\substack{z\in\sons{x_Q^{(k)}}\\z\neq x_Q^{(k-1)}}}\|\Diff_{\anc{\cl Q}{k}}(b_{1,Q})\one_{\Sec_z}\|_\infty\|g_2^{(k)}\one_{\Sec_z}\|_1\\
    &\lesssim \nu_{|x_Q|-k+1}^{|x_Q|}\|b_{1,Q}\|_1\|g_2\|_{\BMO}\Bigl(1+\frac 1{q(x_Q^{(k)})}\sum_{\substack{z\in\sons{x_Q^{(k)}}\\z\neq x_Q^{(k-1)}}}1\Bigr)\\
    &\lesssim \nu_{|x_Q|-k+1}^{|x_Q|}\|b_{1,Q}\|_1\|g_2\|_{\BMO}.
    \end{align*}
Summing over $k$ yields 
    \[\sum_{Q\in\mathcal{F}^1}\langle\Prest(b_{1,Q}) ,g_2\rangle\lesssim \sum_{Q\in\mathcal{F}^1}\|b_{1,Q}\|_1\|g_2\|_{\BMO}\sum_{k=1}^{|x_Q|}\nu_{|x_Q|-k+1}^{|x_Q|}\lesssim  \|f_1\|_{1,Q_0}\|f_2\|_{1,Q_0}\mu(Q_0).\]
We are left with term IV, that we need to further split as follows:
\begin{equation*}
    \langle \Prest b_{1}, b_{2}\rangle
    =\sum_{Q,R\in\mathcal{F}^1}\langle \Prest b_{1,Q}, b_{2,R}\rangle
    \le 2\sum_{Q\in\mathcal{F}^1}\sum_{\substack{R\in\mathcal{F}^1 \\ \mu(R)\le \mu(Q)}}\langle \Prest b_{1,Q}, b_{2,R}\rangle
    = 2\sum_{Q\in\mathcal{F}^1} \text{IV}_Q,
\end{equation*}
where
\begin{align*}
    \text{IV}_Q=&\sum_{x\in\Tr} \Prest(b_{1,Q})(x) \sum_{\substack{R\in\mathcal{F}^1 \\ \mu(R)\le \mu(Q)}} b_{2,R}(x)\mu(x)\\
    =&\sum_{x\in Q^{(1)}} \Prest(b_{1,Q})(x) \sum_{\substack{R\in\mathcal{F}^1 \\ \mu(R)\le \mu(Q)}} b_{2,R}(x)\mu(x)\\
    &+\sum_{k=2}^{|x_Q|}\sum_{x\in Q^{(k)}\setminus Q^{(k-1)}} \Prest(b_{1,Q})(x) \sum_{\substack{R\in\mathcal{F}^1 \\ \mu(R)\le \mu(Q)}} b_{2,R}(x)\mu(x)\\
    =&:\text{IV}_{Q,1}+\sum_{k=2}^{|x_Q|}\text{IV}_{Q,k}.
\end{align*}
For $k\ge2$ we have $b_{2,R}(x_Q^{(k)}) = 0$ because $\mu(R) \leq \mu(Q)$. Therefore, 
\begin{align*}
    |&\text{IV}_{Q,k}|=\Big|\sum_{x\in Q^{(k)}\setminus Q^{(k-1)}} \Prest(b_{1,Q})(x) \sum_{\substack{R\in\mathcal{F}^1 \\ \mu(R)\le \mu(Q)}} b_{2,R}(x)\mu(x)\Big|\\
    & =\Big|\sum_{x\in Q^{(k)}\setminus (Q^{(k-1)}\cup\{x_Q^{(k)}\})} \sum_{y\in Q^{(1)}}(\Krest(x,y)-\Krest(x,\Or Q)) b_{1,Q}(y)\mu(y) \sum_{\substack{R\in\mathcal{F}^1 \\ \mu(R)\le \mu(Q)}} b_{2,R}(x)\mu(x)\Big|\\
    & \le \|b_{1,Q}\|_1 \sup_{y\in Q^{(1)}} \sum_{x\in Q^{(k)}\setminus (Q^{(k-1)}\cup\{x_Q^{(k)}\})} |\Krest(x,y)-\Krest(x,\Or Q)| \sum_{\substack{R\in\mathcal{F}^1 \\ \mu(R)\le \mu(Q)}} |b_{2,R}(x)|\mu(x)\\
    & \leq \|b_{1,Q}\|_1 \Bigg( \sum_{\substack{Q^{(k-1)}\ne S=\cl{S} \in \D \\ S^{(1)}=Q^{(k)}}} \sum_{x\in S}\sup_{y\in Q^{(1)}} |\Krest(x,y)-\Krest(x,\Or Q)| \sum_{\substack{R\in\mathcal{F}^1\!, \;R\subsetneq S \\ \mu(R)\le \mu(Q)}} |b_{2,R}(x)|\mu(x)\Bigg)\\
    & \lesssim\|b_{1,Q}\|_1 \Bigg( \sum_{\substack{Q^{(k-1)}\ne S=\cl{S} \in \D \\ S^{(1)}=Q^{(k)}}} 
    \frac{1}{\mu(Q^{(k-1)})}\nu_{|x_Q|-k}^{|x_Q|}
    \sum_{x\in S}  \sum_{\substack{R\in\mathcal{F}^1\!,\; R\subsetneq S \\ \mu(R)\le \mu(Q)}} |b_{2,R}(x)|\mu(x)\Bigg)\\
    &\le \|b_{1,Q}\|_1 \nu_{|x_Q|-k}^{|x_Q|} \Bigg(\sum_{\substack{Q^{(k-1)}\ne S=\cl{S} \in \D \\ S^{(1)}=Q^{(k)}}}   \avg{f_2}{S} \Bigg) \lesssim \|b_{1,Q}\|_1 \nu_{|x_Q|-k+1}^{|x_Q|}\avg{f_2}{Q_0},
\end{align*}
by Proposition \ref{prop:Krest} and the maximality of the \CZ cubes $Q$. For $\text{IV}_{Q,1}$ we use \eqref{eq: final Prest} to get
\begin{align*}
    \text{IV}_{Q,1}& = \; \sum_{x\in Q^{(1)}} \Prest(b_{1,Q})(x) \sum_{\substack{R\in\mathcal{F}^1 \\ \mu(R)\le \mu(Q)}} b_{2,R}(x)\mu(x)\\
     =& \; \sum_{x\in Q^{(1)}} \sum_{k=1}^{|x_Q|}\Diff_{\anc{\cl Q}{k}}b_{1,Q}(x) \sum_{\substack{R\in\mathcal{F}^1 \\ \mu(R)\le \mu(Q)}} b_{2,R}(x)\mu(x)\\
    =& \;\sum_{x\in Q^{(1)}} \Diff_{\anc{\cl Q}{1}}b_{1,Q}(x) \sum_{\substack{R\in\mathcal{F}^1 \\ \mu(R)\le \mu(Q)}} b_{2,R}(x)\mu(x)\\
     & +\sum_{x\in Q^{(1)}} \sum_{k=2}^{|x_Q|}\Diff_{\anc{\cl Q}{k}}b_{1,Q}(x) \sum_{\substack{R\in\mathcal{F}^1 \\ \mu(R)\le \mu(Q)}} b_{2,R}(x)\mu(x)\\
    =: &\; \text{IV}_{Q,1}^a+\text{IV}_{Q,1}^b.
\end{align*}
We next use Proposition \ref{prop: estim D} to get
\begin{align*}
    |\text{IV}_{Q,1}^b|\le& \|b_{1,Q}\|_1 \sum_{k=2}^{|x_Q|} \frac{\nu_{|x_Q|-k+1}^{|x_Q|}}{\mu(Q^{(k-1)})} \sum_{x\in Q^{(1)}} \sum_{\substack{R\in\mathcal{F}^1 \\ \mu(R)\le \mu(Q)}} b_{2,R}(x)\mu(x)\\
    \le& \|b_{1,Q}\|_1 \sum_{k=2}^{|x_Q|} \nu_{|x_Q|-k+1}^{|x_Q|}\avg{f_2}{Q^{(1)}}\\
    \lesssim& \|b_{1,Q}\|_1 \avg{f_2}{Q_0}\sum_{k=2}^{|x_Q|} \nu_{|x_Q|-k+1}^{|x_Q|}\\
    \lesssim&\|b_{1,Q}\|_1 \avg{f_2}{Q_0}.
\end{align*}
To estimate $\text{IV}_{Q,1}^a$ we proceed differently depending on whether $Q=\cl{Q}$ or not. In the latter case, we use Proposition \ref{prop: estim D} and the fact that $\mu(Q) \sim \mu(Q^{(1)})$ to get
\begin{align*}
|\text{IV}_{Q,1}^a|& \le \frac{\nu_{|x_Q|}^{|x_Q|}}{\mu(Q)} \|b_{1,Q}\|_1 \sum_{x\in Q^{(1)}} \sum_{\substack{R\in\mathcal{F}^1 \\ \mu(R)\le \mu(Q)}} |b_{2,R}(x)|\mu(x) \\ 
& \lesssim \|b_{1,Q}\|_1 \frac{1}{\mu(Q^{(1)})} \sum_{\substack{R\in\mathcal{F}^1 \\ R \subsetneq Q^{(1)}}} \|b_{2,R}\|_1 \\
& \lesssim \|b_{1,Q}\|_1 \avg{f_2}{Q^{(1)}} \lesssim \|b_{1,Q}\|_1 \avg{f_2}{Q_0}.
\end{align*}
	If $Q=\cl{Q}$, then we split again as
	\begin{align*}
		\text{IV}_{Q,1}^a&=\sum_{x\in Q} \Diff_{\anc{\cl Q}{1}}b_{1,Q}(x) \sum_{\substack{R\in\mathcal{F}^1 \\ \mu(R)\le \mu(Q)}} b_{2,R}(x)\mu(x)+\sum_{\substack{x\in \anc Q1\setminus Q\\x\neq \Or Q^{(1)}}} \Diff_{\anc{\cl Q}{1}}b_{1,Q}(x) \sum_{\substack{R\in\mathcal{F}^1 \\ \mu(R)\le \mu(Q)}} b_{2,R}(x)\mu(x)\\
	&=:\text{IV}_{Q,1}^{a,1}+\text{IV}_{Q,1}^{a,2}.
	\end{align*}
	 Hence Proposition \ref{prop: estim D} yields
	\begin{align*}
		|\text{IV}_{Q,1}^{a,1}| \le&
		\sum_{x\in Q} |\Diff_{\anc{\cl Q}{1}}b_{1,Q}(x) |\sum_{\substack{R\in\mathcal{F}^1 \\ R^{(1)}=Q^{(1)}}} |b_{2,R}(x)|\mu(x)\\
		\lesssim&  \|b_{1,Q}\|_1 \frac{1}{\mu(Q)} \sum_{x\in Q} \sum_{\substack{R\in\mathcal{F}^1 \\ R^{(1)}=Q^{(1)}}} |b_{2,R}(x)|\mu(x)\\
		\le& \|b_{1,Q}\|_1\Biggl( \sum_{\substack{S\in\D_1(Q^{(1)}) \\ S\ne Q}}\avg{f_2\one_S}{Q^{(1)}}+\avg{f_2}{Q}\Biggr)\\
		\le&  \|b_{1,Q}\|_1\bigl(\avg{f_2}{Q^{(1)}}+\avg{f_2}{Q}\bigr)\lesssim  \|b_{1,Q}\|_1\bigl(\avg{f_2}{Q_0}+\avg{f_2}{Q}\bigr).
	\end{align*}
Finally, for the last term, again by applying Proposition~\ref{prop: estim D} and using the fact that if $\mu(R)\le\mu(Q)$ and $Q^{(1)}=R^{(1)}$ then $R=\cl{R}$,
\begin{align*}
	&|\text{IV}_{Q,1}^{a,2}| \le \sum_{\substack{x\in Q^{(1)}\setminus Q\\
			x\neq\Or Q^{(1)}}} |\Diff_{\anc{\cl Q}{1}}b_{1,Q}(x) |\sum_{\substack{R\in\mathcal{F}^1 \\ \mu(R)\le \mu(Q)}} |b_{2,R}(x)|\mu(x)\\
	&\lesssim  \|b_{1,Q}\|_1\frac{1}{\mu(Q)} \frac{1}{q(\Or Q^{(1)})}\sum_{\substack{x\in Q^{(1)}\setminus Q\\
			x\neq\Or Q^{(1)}}} \sum_{\substack{R\in\mathcal{F}^1 \\ \mu(R)\le \mu(Q)}} |b_{2,R}(x)|\mu(x)\\
	&\le \|b_{1,Q}\|_1 \frac{1}{\mu(Q)} \frac{1}{q(\Or Q^{(1)})}\Bigg( \sum_{\substack{S\in\D_1(Q^{(1)})\\S\notin\mathcal{F}^1}} \sum_{\substack{R\in\mathcal{F}^1 \\ R^{(1)}\subseteq S}} \sum_{x\in S} |b_{2,R}(x)|\mu(x) + \sum_{\substack{R=\cl{R}\in\mathcal{F}^1 \\ R^{(1)}=Q^{(1)}}} \sum_{x\in Q^{(1)}\setminus Q} |b_{2,R}(x)|\mu(x)\Bigg)\\
	&\le \|b_{1,Q}\|_1\Bigg(\frac{1}{q(\Or Q^{(1)})}
	\sum_{\substack{S\in\D_1(Q^{(1)})\\S\notin \mathcal{F}^1}} \avg{\sum_{\substack{R\in\mathcal{F}^1 \\ R^{(1)}\subseteq S}}|b_{2,R}|}{S}+\sum_{\substack{R=\cl{R}\in\mathcal{F}^1 \\ R^{(1)}=Q^{(1)}}} \frac{1}{q(\Or Q^{(1)})} \frac{\|b_{2,R}\|_1}{\mu(Q)} \Bigg)\\
	&\lesssim \|b_{1,Q}\|_1 \avg{f_2}{Q_0} \Bigg(\frac{1}{q(\Or Q^{(1)})}\sum_{\substack{S\in\D_1(Q^{(1)})\\S\notin \mathcal{F}^1}}1\Bigg) + \|b_{1,Q}\|_1 \frac{1}{q(\Or Q^{(1)})}  \sum_{\substack{R=\cl{R}\in\mathcal{F}^1 \\ R^{(1)}=Q^{(1)}}} \avg{f_2}{R}\\
	&\le \|b_{1,Q}\|_1 \Bigg(\avg{f_2}{Q_0}  + \frac{1}{q(\Or Q^{(1)})}  \sum_{\substack{R=\cl{R}\in\mathcal{F}^1 \\ R^{(1)}=Q^{(1)}}} \avg{f_2}{R}\Bigg).
\end{align*}
	So, when $Q$ is a sector, we control the term $\text{IV}_{Q,1}$ by
	\[|\text{IV}_{Q,1}|\lesssim \|b_{1,Q}\|_1\Big(\avg{f_2}{Q_0} +\avg{f_2}{Q} + \frac{1}{q(\Or Q^{(1)})}  \sum_{\substack{R=\cl{R}\in\mathcal{F}^1 \\ R^{(1)}=Q^{(1)}}} \avg{f_2}{R}\Big).\]
Collecting all the estimates together yields
\begin{align*}
    |\text{IV}_Q|& \leq|\text{IV}_{Q,1}|+\sum_{k=2}^{|x_Q|} |\text{IV}_{Q,k}|\\
    & \lesssim \|b_{1,Q}\|_1\Bigg(\avg{f_2}{Q_0}\Big(1+\sum_{k=2}^{|x_Q|} \nu_{|x_Q|-k+1}^{|x_Q|}\Big) + \avg{f_2}{Q}+ \frac{1}{q(\Or Q^{(1)})}  \sum_{\substack{R=\cl{R}\in\mathcal{F}^1 \\ R^{(1)}=Q^{(1)}}} \avg{f_2}{R}\Bigg)\\
    & \lesssim \|b_{1,Q}\|_1 \avg{f_2}{Q_0} + \|b_{1,Q}\|_1\avg{f_2}{Q} +\|b_{1,Q}\|_1 \frac{1}{q(\Or Q^{(1)})}  \sum_{\substack{R=\cl{R}\in\mathcal{F}^1 \\ R^{(1)}=Q^{(1)}}} \avg{f_2}{R},
\end{align*}
whenever $Q=\cl{Q}$, and
$$
|\text{IV}_Q| \lesssim \|b_{1,Q}\|_1 \avg{f_2}{Q_0}
$$
in the singleton case. Finally, we obtain
\begin{align*}
    |\text{IV}|\lesssim& \sum_{Q\in\mathcal{F}^1}|\text{IV}_Q| \\
     \lesssim &  \sum_{Q\in\mathcal{F}^1} \|b_{1,Q}\|_1 \big(\avg{f_2}{Q_0}+ \avg{f_2}{Q}\big) + \sum_{Q=\cl{Q}\in\mathcal{F}^1} \|b_{1,Q}\|_1 \frac{1}{q(\Or Q^{(1)})}  \sum_{\substack{R=\cl{R}\in\mathcal{F}^1 \\ R^{(1)}=Q^{(1)}}} \avg{f_2}{R}\\
    \le& \avg{f_1}{Q_0} \avg{f_2}{Q_0} \mu(Q_0) + \sum_{Q\in\mathcal{F}^1} \avg{f_1}{Q} \avg{f_2}{Q} \mu(Q) \\
    & + \sum_{Q=\cl{Q}\in\mathcal{F}^1} \avg{f_1}{Q}  \bigg(\frac{1}{q(\Or Q^{(1)})} \sum_{\substack{R=\cl{R}\in\mathcal{F}^1 \\ R^{(1)}=Q^{(1)}}} \avg{f_2}{R}\bigg)\mu(Q).
\end{align*}
This concludes the estimate for IV and yields \eqref{eq: claim first step}, which in turn implies 	
\begin{align*}
|\langle \Proj^{Q_0} f_1,f_2\rangle|
\leq &C\Bigg(\avg{f_1}{Q_0} \avg{f_2}{Q_0} \mu(Q_0) + \sum_{Q\in\mathcal{F}^1} \avg{f_1}{Q} \avg{f_2}{Q} \mu(Q) \\
& + \sum_{Q=\cl{Q}\in\mathcal{F}^1} \avg{f_1}{Q}  \bigg(\frac{1}{q(\Or Q^{(1)})} \sum_{\substack{R=\cl{R}\in\mathcal{F}^1 \\ R^{(1)}=Q^{(1)}}} \avg{f_2}{R}\bigg)\mu(Q)\Bigg)+\sum_{R\in\mathcal{F}^1}\langle \Proj^R(f_1\one_R),f_2\one_R\rangle,
\end{align*}
because $\supp(\Proj^Rf)\subseteq R$ for each $R\in \D$. On the other hand, 
\begin{align*}
\mu(\Omega^0) & \leq \mu(\Omega^0_1)+\mu(\Omega^0_2) \\
& \leq \|\M_\D\|_{L^1(\Tr)\to L^{1,\infty}(\Tr)}\Bigl(\frac{\|f_1\|_1}{4\|\M_\D\|_{L^1(\Tr)\to L^{1,\infty}(\Tr)}\avg{ f_1}{Q_0}}+ \frac{\|f_2\|_1}{4\|\M_\D\|_{L^1(\Tr)\to L^{1,\infty}(\Tr)} \avg{f_2}{Q_0}}\Bigr)\\
& =\frac{\mu(Q_0)}{2}.
\end{align*}
We may then declare that $Q_0 \in \Ss$ and choose $E_{Q_0}=Q_0\setminus \Omega^0$, because we have just proved that $\mu(E_{Q_0})\geq \frac12\mu(Q_0)$. The proof then ends by iteration. Indeed, for every $R\in\mathcal{F}^1$ we can now estimate
\begin{align*}
|\langle \Proj^R(f_1\one_R) & ,f_2\one_R\rangle|
\leq C\Bigg(\avg{f_1}{R} \avg{f_2}{R} \mu(R) + \sum_{T\in\mathcal{F}^1(R)} \avg{f_1}{T} \avg{f_2}{T} \mu(T) \\
& + \sum_{T=\cl{T}\in\mathcal{F}^2(R)} \avg{f_1}{T}  \bigg(\frac{1}{q(\Or T^{(1)})} \sum_{\substack{S=\cl{S}\in\mathcal{F}^2(R) \\ S^{(1)}=T^{(1)}}} \avg{f_2}{S}\bigg)\mu(T)\Bigg)+\sum_{T\in\mathcal F^2(R)}\langle \Proj^T(f_1\one_T),f_2\one_T\rangle,
\end{align*}
applying \eqref{eq: claim first step}. We then proclaim that $\mathcal{F}^1 \subseteq \Ss$ and construct the corresponding sets $E_R$ for each $R\in\mathcal{F}^1$ as before. The procedure can be iterated and yields the assertion after summing all the estimates for each cube.
\end{proof}

\begin{remark}
We have chosen to write down the proof of Theorem \ref{th:theoremA} specifically for $\Proj$, but we could have formulated our result in a much more general way. Clearly, the iterative scheme of proof and the construction of $\Ss$ do not rely on the specific definition of $\Proj$ and could have been carried over with any operator. Then, the core of the proof consists in estimating the action of the truncations $\Proj_{\mathcal{F}^1}^{Q_0}$ over the parts of the \CZ decomposition of $f_1$ and $f_2$. Those estimates in turn rely on uniform $L^2$-boundedness of truncations (good-good term), which can be formulated in a general way, and kernel estimates (good-bad and bad-bad terms). The kernel estimates and the form $\Extras$ are what is specific of $\Proj$ and of the way in which we choose to split our operator into scales. It is absolutely conceivable that the same scheme of proof can be applied to the study of other operators whose structure is tied to that of a martingale filtration in any filtered space, finding wider applicability. In the nondoubling situation, however, one should expect that additional terms like $\Extras$ play a role in the study of most sparse domination results to compensate the lack of regularity of the underlying filtration.  
\end{remark}

\begin{remark}
Theorem \ref{th:theoremA} recovers the standard sparse domination result for $\Proj$ when $\qf:\Tr\to\N$ is bounded. Indeed, in that case $\mu(Q^{(1)}) \leq C_{\mathrm{doub}} \mu(Q)$ for all $Q$ and so
\begin{align*}
\Extras(f_1,f_2) & =\sum_{Q\in \Ss} \avg{f_1}{Q} \frac{1}{q(x_Q^{(1)})} \left(\sum_{\substack{R=\cl{R}\in\Ss \\ R^{(1)}=Q^{(1)}}}\avg{f_2}{R}\right)\mu(Q)\\
& \leq \sum_{Q\in \Ss} \avg{f_1}{Q} \frac{1}{q(x_Q^{(1)})} \avg{f_2}{Q^{(1)}} \mu(Q^{(1)})\\
& \lesssim\sum_{Q\in \Ss} \avg{f_1}{Q^{(1)}} \avg{f_2}{Q^{(1)}} \mu(Q^{(1)}). \\
\end{align*}
But 
$$
\mu(E_Q) \geq \frac{1}{2}\mu(Q) \geq \frac{1}{2C_{\mathrm{doub}}} \mu(Q^{(1)}),
$$
which means that if $\Ss$ is sparse, then the family $\tilde{\Ss}=\{Q^{(1)}: Q\in \Ss\}$ is $1/(2C_{\mathrm{doub}})$-sparse |we have called $1/2$-sparse families just sparse in the rest of the paper|, and that the bound given by Theorem \ref{th:theoremA} becomes
$$
\langle \Proj f_1,f_2\rangle \lesssim \langle \As f_1, f_2\rangle + \langle \mathcal{A}_{\tilde{\Ss}} f_1, f_2\rangle.
$$
This is a sparse estimate, because the union of two sparse families is sparse.
\end{remark}

\subsection{Discrete Bekoll\'e-Bonami weights} \label{sec3} We now study the weighted inequalities that follow from Theorem \ref{th:theoremA}. In the doubling setting, the natural class of weights is the (dyadic) Muckenhaupt one: for $1<p<\infty$, set
\begin{equation*}
    \BB_p(\mu):=\Bigl\{w\colon\Tr\to (0,\infty)\colon [w]_{\BB_p(\mu)}:=\sup_{Q\in\D}\avg{w}{Q}\avg{w^{-\frac{p'}{p}}}{Q}^{\frac{p}{p'}}<\infty\Bigr\}.
\end{equation*}
$\As$ satisfies weighted inequalities for weights belonging to the Muckenhaupt class. Indeed, for $w\in\BB_p(\mu)$, $\As$ is bounded on $L^p(\Tr,w)$ and (see, for example, \cite{CMP2012}) 
\begin{equation*}
    \|\As\|_{L^p(\Tr,w)\to L^p(\Tr,w)} \lesssim [w]_{\BB_p(\mu)}^{\max\{1,\frac{p'}{p}\}}.
\end{equation*}
Due to the presence of $\Extras$ in the sparse domination, the Muckenhaupt class is too large and so we introduce
\begin{equation*}
    \tBB_p(\mu):=\Bigl\{w\colon\Tr\to (0,\infty)\colon [w]_{\tBB_p(\mu)}:=\sup_{\substack{Q,R\in\D\\Q^{(1)}=R^{(1)}}}\avg{w}{Q}\avg{w^{-\frac{p'}{p}}}{R}^{\frac{p}{p'}}<\infty\Bigr\}.
\end{equation*}
It is clear that $[w]_{\BB_p(\mu)}\leq[w]_{\tBB_p(\mu)}$, and so $\tBB_p(\mu)\subseteq \BB_p(\mu)$. In the proof of Theorem \ref{th:theoremC}, we will use the following fact --satisfied by $\tBB_p(\mu)$ weights--: if $E\subseteq S$ and $\mu(E)\sim\mu(S)$, then
\begin{equation} \label{eq:ApSparse}
        w(S)\lesssim [w]_{\BB_p(\mu)} w(E).
\end{equation}

\begin{proof}[Proof of Theorem \ref{th:theoremB}]
Let $1<p<\infty$, $f\in L^p(\Tr,w)$ and $g\in L^{p'}(\Tr,w)$. By Theorem \ref{th:theoremA} we have that there exists a sparse collection $\mathcal{S}$ such that
\begin{equation*}
    \langle \Proj f,gw\rangle\lesssim \langle \As f,gw\rangle+\Extras(f,gw).
\end{equation*}
Furthermore, we already know that $\As$ is bounded on $L^p(\Tr,w)$, so that we only have to deal with $\Extras(f,gw)$:
\begin{align*}
    \Extras(f,gw)
    =&\sum_{Q\in \Ss} \avg{f}{Q} \Bigg(\frac{1}{q(x_Q^{(1)})}\sum_{\substack{R=\cl{R}\in\Ss \\ R^{(1)}=Q^{(1)}}}\avg{gw}{R}\Bigg) \mu(Q) \\
    =& \sum_{Q\in \Ss} \avg{fw^{\frac{p'}{p}}}{Q,w^{-\frac{p'}{p}}} (w^{-\frac{p'}{p}}(Q))^{\frac{1}{p}}\Bigg(\frac{1}{q(x_Q^{(1)})} (w^{-\frac{p'}{p}}(Q))^{1-\frac{1}{p}} \sum_{\substack{R=\cl{R}\in\Ss \\ R^{(1)}=Q^{(1)}}}\avg{gw}{R}\Bigg) \\
    \le & \Bigg(\sum_{Q\in \Ss} \avg{fw^{\frac{p'}{p}}}{Q,w^{-\frac{p'}{p}}}^p w^{-\frac{p'}{p}}(Q)\Bigg)^{\frac{1}{p}}
    \Bigg(\sum_{Q\in \Ss} w^{-\frac{p'}{p}}(Q) \Bigg(\frac{1}{q(x_Q^{(1)})}\sum_{\substack{R=\cl{R}\in\Ss \\ R^{(1)}=Q^{(1)}}} \avg{gw}{R}\Bigg)^{p'}\Bigg)^{\frac{1}{p'}}.
\end{align*}
Concerning the first factor, by \eqref{eq:ApSparse} and the fact that $[w^{-\frac{p'}{p}}]_{\BB_{p'}(\mu)}=[w]_{\BB_p(\mu)}^{\frac{p'}{p}}$, we obtain
\begin{align*}
    \Bigg(\sum_{Q\in \Ss} \avg{fw^{\frac{p'}{p}}}{Q,w^{-\frac{p'}{p}}}^p w^{-\frac{p'}{p}}(Q)\Bigg)^{\frac{1}{p}}
    \lesssim& [w^{-\frac{p'}{p}}]_{\BB_{p'}(\mu)}^{\frac{1}{p}} \Bigg(\sum_{Q\in \Ss} \avg{fw^{\frac{p'}{p}}}{Q,w^{-\frac{p'}{p}}}^p w^{-\frac{p'}{p}}(E_Q)\Bigg)^{\frac{1}{p}} \\
    \le& [w]_{\BB_p(\mu)}^{\frac{p'}{p^2}} \Bigg(\sum_{x\in\Tr} \left(\M_{\D,w^{-\frac{p'}{p}}}(f w^{\frac{p'}{p}})\right)^p\!(x)\,w^{-\frac{p'}{p}}(x)\mu(x) \Bigg)^{\frac{1}{p}} \\
    \le & [w]_{\BB_p(\mu)}^{\frac{p'}{p^2}} \Bigg(\sum_{x\in\Tr} f^p(x) w^{p'-\frac{p'}{p}}(x)\mu(x) \Bigg)^{\frac{1}{p}} \\
    =& [w]_{\BB_p(\mu)}^{\frac{p'}{p^2}} \|f\|_{L^p(\Tr,w)}.
\end{align*}
For the second one we use Jensen's inequality to get
\begin{align*}
    \Bigg(\sum_{Q\in \Ss} w^{-\frac{p'}{p}}(Q) \Bigg(\frac{1}{q(x_Q^{(1)})} &\sum_{\substack{R=\cl{R}\in\Ss \\ R^{(1)}=Q^{(1)}}} \avg{gw}{R}\Bigg)^{p'}\Bigg)^{\frac{1}{p'}}
    \\
     = &\;\Bigg(\sum_{Q\in \Ss} \frac{w^{-\frac{p'}{p}}(Q)}{\mu(Q)^{p'}} \Bigg(\frac{1}{q(x_Q^{(1)})}\sum_{\substack{R=\cl{R}\in\Ss \\ R^{(1)}=Q^{(1)}}} \avg{gw}{R,w}w(R)\Bigg)^{p'}\Bigg)^{\frac{1}{p'}} \\
      \leq &\;\Bigg(\sum_{Q\in \Ss} \frac{w^{-\frac{p'}{p}}(Q)}{\mu(Q)^{p'}} \frac{1}{q(x_Q^{(1)})}\sum_{\substack{R=\cl{R}\in\Ss \\ R^{(1)}=Q^{(1)}}} \avg{gw}{R,w}^{p'}w(R)w(R)^{p'-1}\Bigg)^{\frac{1}{p'}} \\
     \lesssim &\; \Bigg(\sup_{\substack{Q,R\in\D \\ Q^{(1)}=R^{(1)}}} \frac{w^{-\frac{p'}{p}}(Q)}{\mu(Q)} \frac{w(R)^{p'-1}}{\mu(R)^{p'-1}} \Bigg)^{\frac{1}{p'}} \\
    & \times \Bigg(\sum_{Q\in \Ss} \frac{1}{q(x_Q^{(1)})}\sum_{\substack{R=\cl{R}\in\Ss \\ R^{(1)}=Q^{(1)}}} \avg{gw}{R,w}^{p'} w(R)\Bigg)^{\frac{1}{p'}}\\
     \le &\;  \Bigg(\sup_{\substack{Q,R\in\D \\ Q^{(1)}=R^{(1)}}} \avg{w^{-\frac{p'}{p}}}{Q}^{\frac{1}{p'}} \avg{w}{R}^{\frac{1}{p}}\Bigg)
    \Bigg(\sum_{Q\in \Ss}  \avg{gw}{Q,w}^{p'} w(Q)\Bigg)^{\frac{1}{p'}} \\
     \lesssim &\; [w]_{\tBB_p(\mu)}^{\frac1p}[w]_{\BB_p(\mu)}^{\frac{1}{p'}} \|g\|_{L^{p'}(\Tr,w)},
\end{align*}
where in the last step we proceeded as we did for the first factor.
Finally, gathering all the estimates together
\begin{equation*}
    \Extras(f,gw)\lesssim [w]_{\tBB_p(\mu)}^{\frac1p}[w]_{\BB_p(\mu)}^{\frac{p'}{p^2}+\frac{1}{p'}} \|f\|_{L^{p}(\Tr,w)} \|g\|_{L^{p'}(\Tr,w)}
\end{equation*}
and then
\begin{equation*}
    \langle \Proj f,gw\rangle \lesssim [w]_{\tBB_p(\mu)}^{\frac1p}[w]_{\BB_p(\mu)}^{\frac{p'}{p^2}+\frac{1}{p'}} \|f\|_{L^{p}(\Tr,w)} \|g\|_{L^{p'}(\Tr,w)}.
\end{equation*}
\end{proof}

\begin{remark}
Our Muckenhaupt-type conditions can be viewed as natural counterparts of the conditions found by Bekoll\'e and Bonami \cite{BB1978} in the continuous setting. Indeed, there the only balls that appear in the condition on the weights are those whose radii are comparable to their distance to the boundary. Using the usual embedding of $\Tr$ in the hyperbolic disk (as in \cite{CC1994}), the sets that correspond to those balls in the tree are precisely the elements of $\D$ which are not singletons, which are the ones that appear in the definition of the class $\tBB_p(\mu)$.
\end{remark}

\begin{remark}
When $q$ is bounded, $\Proj$ admits sparse domination in the usual sense and we recover the linear dependence on $[w]_{\BB_2(\mu)}$ in its $L^2$-bound. This would be in line with the sharp result in \cite{PR2013}. Our proof does not provide that because of $\Extras$, which gives a power of $3/2$ in the $L^2$-case. We suspect that this is not an artifact of the proof but we leave it as an interesting open problem. 
\end{remark}

\section{Endpoint estimates} \label{sec4}

\subsection{General \CZ theory in $\Tr$}

In this section we work with $L^2$-bounded integral operators $T$ of the form
$$
Tf(x) = \sum_{y \in \Tr} K(x,y) f(y) \mu(y), 
$$
for some kernel $K\colon\Tr\times\Tr\to\C$ which satisfies $K(x,y)=K(y,x)$ for all $x,y\in\Tr$. $K$ is said to satisfy H\"ormander's condition if

\begin{equation} \label{eq:HormCond} \tag{H\"ormCond}
    \sup_{Q\in\D}\sup_{x,y\in Q} \sum_{z\in\Tr\setminus Q} |K(z,x)-K(z,y)|\mu(z)<\infty.
\end{equation}
The above is just the classical H\"ormander's condition rewritten in our setting. 
The following estimate can be viewed as an integral size condition in the spirit of the one in \cite{CoPa2019} |already present in \cite{To2001b}, albeit in a more implicit fashion|:
\begin{equation} \label{eq:sizeCond} \tag{SizeCond}
    \sup_{Q \in \D} \sup_{x\in Q}\sum_{z\in Q^{(1)}\setminus Q}|K(x,z)|\mu(z)<\infty.
\end{equation}

\begin{remark}
The symmetry condition $K(x,y)=K(y,x)$ for the kernel can be dropped but then \eqref{eq:HormCond} and \eqref{eq:sizeCond} have to be modified accordingly to adjust the proof below. In particular, $\Hone$ and $L^1$ estimates require the integration to run over the first variable of the kernel, while $\BMO$ ones require that it runs over the second.  
\end{remark}
Throughout this subsection, we do not need to assume that $\Tr$ is radial, nor do we need any assumption on the measure $\mu$. 

\begin{proof}[Proof of Theorem \ref{th:theoremC}] We proof the three claimed estimates in turn. 

\textbf{Weak-type bound.} Let $0\leq f \in L^1$. If $0<\lambda \leq \|f\|_1$, then
$$
\lambda \mu\left(\{x\in\Tr: |Tf(x)|>\lambda \}\right) \leq \lambda \leq \|f\|_1,
$$
and there is nothing to prove. Otherwise, fix $\lambda>\|f\|_1$. We apply Lemma \ref{lem:CZnonhomogeneous} to the \CZ cubes $\{Q_j\}_j$ |that is, to the smallest disjoint cubes in $\D$ that cover $\{\M_\D f > \lambda\}$|. Denote $\Omega_\lambda = \cup_j Q_j$ and write
		\begin{align*}
			\mu\Bigl(\Bigl\{x\in \Tr: |T f(x)| > \lambda \Bigr\}\Bigr)  \leq & \; \mu(\Omega_\lambda) + \mu\Bigl(\Bigl\{x\in \Tr: |T g(x)| >\frac \lambda2 \Bigr\}\Bigr) \\
			& + \mu\Bigl(\Bigl\{x\in \Tr\setminus\Omega_\lambda: |T b(x)| >\frac \lambda2 \Bigr\}\Bigr)\\
			 =: & \; \mathrm{I}+\mathrm{II}+\mathrm{III}.
		\end{align*}
The weak-type of $\M_\D$ immediately gives 
\begin{equation*}
    \mathrm{I} \lesssim \frac{\|f\|_1}{\lambda}.
\end{equation*}
Chebyshev's inequality, the $L^2$-boundedness of $T$, and property (3) of Lemma \ref{lem:CZnonhomogeneous} yield
\begin{equation*}
    \mathrm{II} \lesssim \frac{1}{\lambda^2} \|T g \|_2^2 \lesssim \frac{1}{\lambda^2} \| g \|_2^2 \lesssim \frac{1}{\lambda} \|f\|_1.
\end{equation*}
Finally, by Chebyshev's inequality again,
		\begin{align*}
			\mathrm{III} & \leq \frac{1}{\lambda} \sum_{x\in\Tr\setminus\Omega_\lambda} |T b(x)| \mu(x)
			\leq \frac{1}{\lambda} \sum_j \sum_{x\in \Tr\setminus Q_j} |T b_j(x)| \mu(x)
			=: \frac{1}{\lambda} \sum_j \mathrm{III}_j.
		\end{align*}
To estimate each term, we first split into two regions:
		\begin{align*}
			\mathrm{III}_j & = \sum_{x\in \Tr\setminus Q_j} \Bigl| \sum_{y\in\anc {Q}1_j} K(x,y) b_j(y) \mu(y)\Bigr| \mu(x) \\
			& = \sum_{x\in \Tr\setminus\anc{Q}{1}_j} \Bigl| \sum_{y\in \anc {Q}1_j} K(x,y) b_j(y) \mu(y)\Bigr| \mu(x)
			+ \sum_{x\in \anc {Q}{1}_j\setminus Q_j} \Bigl| \sum_{y\in \anc {Q}{1}_j} K(x,y) b_j(y) \mu(y)\Bigr| \mu(x).
		\end{align*}
For the first one, we use the mean zero of $b_j$ and~\eqref{eq:HormCond}:
		\begin{align*}
		 \sum_{x\in \Tr\setminus\anc{Q}{1}_j} \Bigl| \sum_{y\in \anc {Q}1_j} &K(x,y) b_j(y) \mu(y)\Bigr| \mu(x) = \sum_{x\in \Tr\setminus\anc{Q}{1}_j}  \Bigl| \sum_{y\in \anc {Q}1_j} (K(x,y)-K(x,\Or{Q_j})) b_j(y) \mu(y)\Bigr| \mu(x) \\
			& \leq \sum_{y\in \anc {Q}1_j} |b_j(y)| \Bigl( \sup_{z\in \anc {Q}1_j} \sum_{x\in \Tr\setminus\anc{Q}{1}_j} \left|  K(x,y)-K(x,z) \right|\mu(x) \Bigr)\mu(y)
			\lesssim \|b_j\|_1.
		\end{align*}
For the remaining term, we further use the definition of $b_j$:
		\begin{align*}
			 \sum_{x\in \anc {Q}{1}_j\setminus Q_j} &\Bigl| \sum_{y\in \anc {Q}{1}_j} K(x,y) b_j(y) \mu(y)\Bigr| \mu(x) \\
			&\leq \sum_{x\in \anc Q1_j\setminus Q_j} \Bigl| \sum_{y\in Q_j} K(x,y) b_j(y) \mu(y)\Bigr| \mu(x)
			+ \sum_{x\in \anc Q1_j} \Bigl|  T\Bigl(\langle f\one_{Q_j}\rangle_{\anc Q1_j}\one_{\anc Q1_j}\Bigr)(x)\Bigr| \mu(x) \\
			& =: \alpha_j + \beta_j.
		\end{align*}
		We deal with $\beta_j$ using H\"older's inequality and the $L^2$-bound:
		\begin{align*}
			\beta_j & \leq \mu(\anc Q1_j)^{\frac12} \left\| T\left(\langle f \one_{Q_j}\rangle_{\anc Q1_j} \one_{\anc Q1_j} \right)\right\|_2
			\lesssim \mu(\anc Q1_j)^{\frac12} \left\|\langle f \one_{Q_j}\rangle_{\anc Q1_j} \one_{\anc Q1_j} \right\|_2 =\|f\one_{Q_j}\|_1.
		\end{align*}
For $\alpha_j$ we use Fubini's theorem and \eqref{eq:sizeCond}:
		\begin{align*}
			\alpha_j&\leq \sum_{x\in \anc Q1_j\setminus Q_j}  \sum_{y\in Q_j} |K(x,y)| |f(y)| \mu(y)\mu(x)\\
			& \leq \sum_{y\in Q_j} |f(y)| \mu(y) \sup_{y' \in Q_j} \sum_{x\in\anc Q1_j \setminus Q_j} |K(x,y')| \mu(x) \lesssim \|f \one_{Q_j}\|_1. 
		\end{align*}
		Finally, we collect all estimates together and by property (2) of Lemma \ref{lem:CZnonhomogeneous},
		\begin{equation*}
			\mathrm{III} \lesssim \frac{1}{\lambda} \sum_j \left( \|b_j\|_1 + \|f \one_{Q_j}\|_1\right)\lesssim \frac{1}{\lambda} \|f\|_1,
		\end{equation*}
		as desired.

\textbf{$L^\infty-\BMO$ bound.} We use the characterization of the $\BMO$ norm given in \eqref{eq:equivnorm}. Let $f\in L^\infty(\Tr)$. The estimate
$$
\sup_{Q\in\D} \frac{1}{\mu(Q)}\sum_{x\in Q} |Tf(x)-\avg{Tf}{Q}| d\mu \lesssim \|f\|_\infty 
$$
is standard and we omit its proof. On the other hand, for each $Q\in \D$ we use the kernel representation to write
\begin{align*}
    \left|\avg{Tf}{Q}- \avg{Tf}{Q^{(1)}} \right| =& \frac{1}{\mu(Q)\mu(Q^{(1)})} \left|\sum_{x\in Q}Tf(x)\mu(x) -\sum_{z\in Q^{(1)}} Tf(z) \mu(z)\right| \\
    =& \frac{1}{\mu(Q)\mu(Q^{(1)})} \left|\sum_{\substack{x\in Q\\z\in Q^{(1)}}}\sum_{y \in \Tr} (K(x,y)-K(z,y))f(y)\mu(y)\mu(x)\mu(z)\right| \\ 
    \leq& \sup_{\substack{x\in Q\\ z\in Q^{(1)}}} \sum_{y \in \Tr \setminus Q^{(1)}} \left|K(x,y)-K(z,y))\right||f(y)|\mu(y) \\
    & + \avg{|T(f\one_Q)|}{Q} + \avg{|T(f\one_{Q^{(1)}})|}{Q^{(1)}} + \sup_{x\in Q} \sum_{y\in Q^{(1)}\setminus Q} |K(x,y)||f(y)|\mu(y) \\
    =:& \mathrm{I} + \mathrm{II} + \mathrm{III} + \mathrm{IV}.\\
\end{align*}
The estimate $\mathrm{I} \lesssim \|f\|_\infty$ follows directly from \eqref{eq:HormCond}. Next, by the $L^2$-bound for $T$
$$
\mathrm{II} \leq\mu(Q)^{-\frac12}\|Tf\|_{L^2(Q)} \lesssim \mu(Q)^{-\frac12}\|f\|_{L^2(Q)} \lesssim \|f\|_\infty,
$$
and a similar estimate holds for $\mathrm{III}$. Finally, \eqref{eq:sizeCond} implies $\mathrm{IV} \lesssim \|f\|_\infty$ and completes the proof.

\textbf{$\Hone-L^1$ bound.} It is enough to prove that $\|Tb\|_1\lesssim 1$ for every simple atomic block $b$. Fix one such $b$, so that 
$$
b=a - \avg{a}{Q}\one_Q,
$$
where $\supp(a) \subseteq R \in \D_1(Q)$ and $\|a\|_\infty \leq \mu(R)^{-1}$. We write
$$
\|Tb\|_1 \leq \|Tb\|_{L^1(Q)} + \|Tb\|_{L^1(\Tr\setminus Q)},
$$
and handle each term in turn. For the first one, we further split
$$
\|Tb\|_{L^1(Q)} \leq \|Ta\|_{L^1(R)} +\|Ta\|_{L^1(Q\setminus R)}+ |\avg{a}{Q}|\|T1_{Q}\|_{L^1(Q)}.
$$
The first term on the right hand side above is dealt with using the $L^2$-bound of $T$ and the size estimate on $a$:
$$
\|Ta\|_{L^1(R)} \leq \mu(R)^{\frac12} \|Ta\|_{L^2(R)} \lesssim \mu(R)^{\frac12} \|a\|_{L^2(R)} \leq 1.
$$
Similarly, for the third term on the right hand side above we obtain
$$
|\avg{a}{Q}|\|T1_{Q}\|_{L^1(Q)} \leq \|a\|_\infty \frac{\mu(R)}{\mu(Q)}\|T1_{Q}\|_{L^2(Q)}\mu(Q)^{\frac12} \lesssim 1,
$$
using $\supp(a) \subseteq R$. The second term on the right hand side above requires \eqref{eq:sizeCond}:
\begin{align*}
\|Ta\|_{L^1(Q\setminus R)} & \leq \sum_{x \in Q\setminus R} \sum_{y\in R} |K(x,y)||a(y)|\mu(y)\mu(x)\\
& \leq \|a\|_1 \sup_{y\in R}\sum_{x \in Q\setminus R}|K(x,y)|\mu(x)\\
& \lesssim \|a\|_1 \leq \|a\|_\infty \mu(R) \leq 1.
\end{align*}
Finally, using \eqref{eq:HormCond} and the mean zero of $b$ in $Q$ yields
\begin{align*}
    \|Tb\one_{\Tr\setminus Q}\|_1 &=\sum_{x\in\Tr\setminus Q} |\sum_{y\in Q} K(x,y)b(y)\mu(y)|\mu(x)\\
    &=\sum_{x\in\Tr\setminus Q} \big|\sum_{y\in Q} (K(x,y)-K(x,x_Q))b(y)\mu(y)\big|\mu(x)\\
    &\le\|b\|_1 \sup_{y\in Q} \sum_{x\in\Tr\setminus Q}\big|K(x,y)-K(x,v_Q)\big|\mu(x)\\
    &\lesssim \|b\|_1 \leq \|b\|_{\Hone} \leq 1,\\ 
\end{align*}
which finishes the proof.
\end{proof}

\subsection{Endpoint results for $\Proj$} We start checking the H\"ormander and size conditions for the kernel of $\Proj$. As one could have expected, they are weaker than the estimates in Section \ref{sec1} and can be deduced from them. For instance, \eqref{eq:HormCond} can be checked as follows: fix a non-singleton $Q\in\D$ |otherwise, there is nothing to prove| and $x\neq y \in Q$. Summing in annuli and using Proposition \ref{prop:Krest} yields
\begin{align*}
    \sum_{z\in \Tr\setminus Q} |\K &(z,x) -\K(z,y)| \mu(z) = \sum_{\ell=1}^{|x_Q|} \sum_{z\in Q^{(\ell)}\setminus Q^{(\ell-1)}} |\K(z,x)-\K(z,y)| \mu(z) \\
    & \leq \sum_{\ell=1}^{|x_Q|} \mu(x_Q^{(\ell)}) \frac{\nu_{|x_Q|-\ell+1}^{|x_Q|}}{\mu(Q^{(\ell)})} + \sum_{\ell=1}^{|x_Q|} \mu(Q^{(\ell)}\setminus (Q^{(\ell-1)} \cup \{x_Q^{(\ell)}\})) \frac{\nu_{|x_Q|-\ell}^{|x_Q|}}{\mu(Q^{(\ell-1)})} \\
    & \lesssim \sum_{\ell=1}^{|x_Q|} \nu_{|x_Q|-\ell+1}^{|x_Q|} + \sum_{\ell=1}^{|x_Q|} \mu(Q^{(\ell-1)})q(x_Q^{(\ell)})  \frac{\nu_{|x_Q|-\ell+1}^{|x_Q|}}{q(x_Q^{(\ell)})\mu(Q^{(\ell-1)})} \lesssim 1.
\end{align*}
A similar computation using Lemma \ref{lem: KrestF}, part (2) gives \eqref{eq:sizeCond} for $\K$. As a consequence, we can apply Theorem \ref{th:theoremC} to $\Proj$ and deduce the corresponding endpoint estimates. We finish proving the stronger $\Hone-\Hone$ and $\BMO-\BMO$ endpoint bounds.

\begin{proof}[Proof of Theorem \ref{th:theoremD}] Even if the $\Hone$ estimate implies the $\BMO$ one, we prove both directly, without using duality.

\textbf{$\BMO\to\BMO$ estimate:} Let $f\in\BMO$. Since $\Proj\one_\Tr=\one_\Tr$
        \begin{equation*}
            \left|\sum_{x\in\Tr} \Proj f(x)\mu(x)\right|
            =\left|\sum_{z\in\Tr} f(z)\mu(z)\right|
            \le \|f\|_\BMO,
        \end{equation*}
and we just have to bound the two suprema in~\eqref{eq:equivnorm}. Fix $Q\in\D$ and split $f$ as $f=f-\avg{f}{Q^{(1)}}+\avg{f}{Q^{(1)}}=:\tilde f +\avg{f}{Q^{(1)}}$. Since $\Proj$ sends constants to constants, it is enough estimate the $\BMO$ norm of $\Proj \tilde f$. Since $\tilde f$ has mean $0$ over $Q^{(1)}$, we have 
        \begin{equation}\label{eq:tildemeans}
            |\langle \tilde f \rangle_{Q^{(\ell)}}| + \sup_{R\in \D_\ell(Q)} |\langle \tilde f \rangle_{R}| \lesssim (\ell +1) \|f\|_{\BMO}, \quad \ell \geq 0. 
        \end{equation}
To estimate the norm of $\Proj \tilde f$, we study separately the two suprema in \eqref{eq:equivnorm}. For the first, we fix $Q\in\D$ and write $\tilde f= \tilde f\one_{Q}+ \tilde f\one_{\Tr\setminus Q}$ and we denote by $\mathrm{I}$ and $\mathrm{II}$ the corresponding associated terms after applying triangle inequality. Concerning the first part, John-Nirenberg inequality and the $L^2$-boundedness of $\Proj$ give
\begin{align*}
    \mathrm{I} & = \frac{1}{\mu(Q)} \sum_{x\in Q} |\Proj(\tilde f\one_Q)(x)-\avg{\Proj(\tilde f\one_Q)}{ Q}|\mu(x)\\
    &\lesssim \left(\frac{1}{\mu(Q)} \sum_{x\in Q} |\Proj(\tilde f\one_Q)(x)|^2\mu(x)\right)^{\frac{1}{2}} \lesssim \left(\frac{1}{\mu(Q)} \sum_{x\in \Tr} |\tilde f\one_Q(x)|^2\mu(x)\right)^{\frac{1}{2}}\\
    &=\left(\frac{1}{\mu(Q)} \sum_{x\in Q} |f(x)-\avg{f}{Q^{(1)}}|^2\mu(x)\right)^{\frac{1}{2}} \le \|f\|_\BMO.
\end{align*}
For $\mathrm{II}$, if $Q$ is a single point there is nothing to do. Otherwise, we decompose it into annuli:
\begin{align*}
    \mathrm{II}&=\frac{1}{\mu(Q)} \sum_{x\in Q} |\Proj(\tilde f\one_{\Tr\setminus Q})(x)-\avg{\Proj(\tilde f\one_{\Tr\setminus Q})}{ Q}|\mu(x)\\
    &=\frac{1}{\mu(Q)} \sum_{x\in Q} \left|\sum_{z\in \Tr\setminus Q} \K(z,x)\tilde f(z)\mu(z)-\frac{1}{\mu(Q)}\sum_{y\in Q}\sum_{z\in \Tr\setminus Q} \K(z,y)\tilde f(z)\mu(z)\mu(y)\right|\mu(x)\\
    &\le\frac{1}{\mu(Q)^2} \sum_{x\in Q}\sum_{y\in Q} \sum_{z\in \Tr\setminus Q} |\K(z,x)-\K(z,y)||\tilde f(z)|\mu(z)\mu(y)\mu(x)\\
    & \leq \sup_{x,y \in Q} \sum_{z\in \Tr\setminus Q} |\K(z,x)-\K(z,y)||\tilde f(z)|\mu(z)\\
    & \leq \sum_{\ell=1}^{|x_Q|} \sup_{x,y \in Q} \sum_{z\in Q^{(\ell)}\setminus Q^{(\ell-1)}} |\K(z,x)-\K(z,y)||\tilde f(z)|\mu(z) =: \sum_{\ell=1}^{|x_Q|} \mathrm{II}_\ell.
\end{align*}
We now estimate each term using Lemma \ref{prop:Krest} and \eqref{eq:tildemeans}:
\begin{align*}
\mathrm{II}_\ell & \lesssim |\tilde{f}(x_Q^{(\ell)})|\mu(x_Q^{(\ell)}) \frac{\nu_{|x_Q|-\ell+1}^{|x_Q|}}{\mu(Q^{(\ell)})} + \frac{\nu_{|x_Q|-\ell}^{|x_Q|}}{\mu(Q^{(\ell-1)})} \sum_{z\in Q^{(\ell)}\setminus (Q^{(\ell-1)} \cup \{x_Q^{(\ell)}\})} |\tilde{f}(z)| \mu(z) \\
& \leq \nu_{|x_Q|-\ell+1}^{|x_Q|}  \avg{|\tilde f|}{Q^{(\ell)}}  + \nu_{|x_Q|-\ell+1}^{|x_Q|} \frac{1}{q(x_Q^{(\ell)})}\sum_{R\in \D_1(Q^{(\ell)})}  \avg{|\tilde f|}{R}  \\
& \leq 2\cdot \nu_{|x_Q|-\ell+1}^{|x_Q|} \avg{|\tilde f|}{Q^{(\ell)}} + \nu_{|x_Q|-\ell+1}^{|x_Q|} \frac{1}{q(x_Q^{(\ell)})}\sum_{R\in \D_1(Q^{(\ell)})} \avg{|\tilde f|}{R}-\avg{|\tilde f|}{Q^{(\ell)}}\\
& \lesssim \ell \cdot\nu_{|x_Q|-\ell+1}^{|x_Q|} \||\tilde f|\|_{\BMO} \lesssim  \ell\cdot\nu_{|x_Q|-\ell+1}^{|x_Q|} \| f\|_{\BMO}.
\end{align*}
The resulting quantity is acceptable due to the geometric decay of $\nu_{|x_Q|-\ell}^{|x_Q|}$. To bound the second term in \eqref{eq:equivnorm}, we split
\begin{align*}
    |\avg{\Proj \tilde f}{Q}-\avg{\Proj \tilde f}{Q^{(1)}}|
    \le& |\avg{\Proj (\tilde f \one_{Q})}{Q}|
    +|\avg{\Proj(\tilde f \one_{Q^{(1)}})}{Q^{(1)}}| \\
    & +|\avg{\Proj (\tilde f\one_{Q^{(1)}\setminus Q})}{Q}|
    +|\avg{\Proj (\tilde f\one_{\Tr\setminus Q^{(1)}})}{Q}-\avg{\Proj (\tilde f\one_{\Tr\setminus Q^{(1)}})}{Q^{(1)}}|.
\end{align*}
The first two terms above are estimated similarly. Indeed, by the $L^2$-boundedness of $\Proj$,
$$
|\avg{\Proj(\tilde f \one_{Q^{(1)}})}{Q^{(1)}}| \lesssim \left(\avg{|\tilde f|^2 }{Q^{(1)}}\right)^\frac{1}{2} = \left(\avg{|\tilde f - \avg{\tilde f}{Q^{(1)}}|^2 }{Q^{(1)}}\right)^\frac{1}{2} \leq \|\tilde{f}\|_{\BMO},
$$
and a similar computation applies to $|\avg{\Proj (\tilde f \one_{Q})}{Q}|$. The third term above is estimated using Lemma \ref{lem: KrestF}, part (2), as follows:
\begin{align*}
    |\langle\Proj (\tilde f & \one_{Q^{(1)}\setminus Q})\rangle_{Q}|  \leq \frac{1}{\mu(Q)} \sum_{x\in Q} \sum_{y \in Q^{(1)}\setminus Q} |\K(x,y)| |\tilde f(y)| \mu(y) \\
    \leq& \sup_{x\in Q}\sum_{y \in Q^{(1)}\setminus Q} |\K(x,y)| |\tilde f(y)| \mu(y) \lesssim \sum_{y \in Q^{(1)}\setminus Q} |\K(x_Q,y)| |\tilde f(y)| \mu(y) \\
    \leq& \sum_{y \in Q^{(1)}\setminus Q} \sum_{\ell=1}^{|x_Q|} |\K_{Q^{(\ell)}}(x_Q,y)| |\tilde f(y)| \mu(y) \\
    \leq &\left(\sum_{\ell=1}^{|x_Q|}\frac{1}{\mu(Q^{(\ell)})}\right) f(x_Q^{(1)}) \mu(x_Q^{(1)}) \\
    & + \left(\frac{1}{q(x_Q^{(1)}) \mu(Q)} + \sum_{\ell=1}^{|x_Q|}\frac{1}{\mu(Q^{(\ell)})}\right) \sum_{y\in Q^{(1)} \setminus (Q \cup \{x_Q^{(1)}\})} |\tilde f (y) | \mu(y)\\
    \lesssim & \avg{|\tilde f|}{Q^{(1)}} + \frac{1}{q(x_Q^{(1)})} \sum_{R \in \D_1(Q^{(1)})} \avg{|\tilde f|}{R} \lesssim \|f\|_{\BMO},
\end{align*}
finishing the computation as with $\mathrm{II}_\ell$ before. Finally, for the last term we write
            \begin{align*}
                |\avg{\Proj (\tilde f\one_{\Tr\setminus Q^{(1)}})}{Q}&-\avg{\Proj (\tilde f\one_{\Tr\setminus Q^{(1)}})}{Q^{(1)}}|\\
                \le & \frac 1{\mu (Q)}\frac 1{\mu (Q^{(1)})}\sum_{x\in Q}\sum_{y\in Q^{(1)}}|\Proj (\tilde f\one_{\Tr\setminus Q^{(1)}})(x)-\Proj (\tilde f\one_{\Tr\setminus Q^{(1)}})(y)|\mu(y)\mu(x)\\
                \le & \frac 1{\mu (Q)}\frac 1{\mu (Q^{(1)})}\sum_{x\in Q}\sum_{y\in Q^{(1)}}\sum_{z\in \Tr\setminus Q^{(1)}}|\tilde f(z)||\K(z,x)-\K(z,y)|\mu(z)\mu(y)\mu(x)\\
                \le & \sup_{x,y \in Q^{(1)}} \sum_{z\in \Tr\setminus Q^{(1)}}|\tilde f(z)||\K(z,x)-\K(z,y)|\mu(z),
            \end{align*}
and we split into annuli as with term $\mathrm{I}$ above to conclude.

\textbf{$\Hone \to \Hone$ estimate:} Let $b\in\Hone$ be a simple atomic block supported on $Q^{(1)} \in \D$, i.e. $b=a-\avg{a}{\anc Q1}1_{\anc Q1}$, with $a$ supported in $Q$ satisfying the size condition. It is enough to check that $\|\Proj b\|_{\Hone} \lesssim 1$, and we do it by finding a decomposition of it into atomic blocks that need not be simple, by the equivalence in Lemma \ref{lem:simpleAtoms}. Because of cancellation, we have
        \begin{align*}
            \Proj b&=\sum_{R\in\D} \Diff_Rb=
            \sum_{\substack{R\in\D\\R\subsetneq \anc Q1}}\Diff_Rb+
        \sum_{\substack{R\in\D\\R\supseteq \anc Q1}}\Diff_Rb\\
            &=
            \sum_{\substack{R\in\D\\ R\subsetneq\anc Q1}}\sum_{\ell}\bigl\langle a-\avg{a}{\anc Q1},h_R^\ell\Bigr\rangle h_R^\ell +\sum_{k=1}^{|Q|} \Diff_{\anc{Q}{k}}b
        =\Proj^{Q}a+\sum_{k=1}^{|Q|} \Diff_{\anc{Q}{k}}b.
        \end{align*}
Clearly, $\supp(\Proj^Qa)\subseteq Q$ and has mean zero, so it is an atomic block. By the $L^2$-boundedness, $\|\Proj^Q a\|_2\lesssim\|a\|_2\leq \mu(Q)^{-1/2}$, so $\|\Proj^Q a\|_{\Hone} \lesssim 1$.

We next claim that, for each $k$, $\Diff_{\anc{Q}{k}}b$ is an atomic block. It is supported on $\anc{Q}{k}$, where it has zero mean, and we can write it as
        \begin{equation*}
            \Diff_{\anc{Q}{k}}b=\one_{\anc{Q}{k-1}}\Diff_{\anc{Q}{k}}b+
            \sum_{\substack{T:\anc{T}{1}=\anc{Q}{k}\\T\ne\anc{Q}{k-1}}} \one_T\Diff_{\anc{Q}{k}}b.
        \end{equation*}
We have
        \begin{align*}
            \|\one_{\anc{Q}{k-1}}\Diff_{\anc{Q}{k}}b\|_2
            &\le \mu(\anc{Q}{k-1})^{\frac{1}{2}}\|\one_{\anc{Q}{k-1}}\Diff_{\anc{Q}{k}}b\|_\infty\\
            &\lesssim \frac{\mu(\anc{Q}{k-1})^{\frac{1}{2}} \nu_{|x_Q|-k+1}^{|x_Q|}}{\mu(\anc{Q}{k-1})}\|b\|_1\\
            &\le \frac{\nu_{|x_Q|-k+1}^{|x_Q|}}{\mu(\anc{Q}{k-1})^{\frac{1}{2}}},
        \end{align*}
        by H\"older's inequality, Proposition \ref{prop: estim D} and the fact that $\|b\|_1\le\|b\|_{\Hone}\le 1$. Analogously we get
        \begin{equation*}
            \|\one_T\Diff_{\anc{Q}{k}}b\|_2\lesssim \frac{\nu_{|x_Q|-k}^{|x_Q|}}{\mu(\anc{Q}{k-1})^{\frac{1}{2}}}
        \end{equation*}
        for every $T\in\D$ such that $T^{(1)}=\anc{Q}{k}$, $T\ne\anc{Q}{k-1}$.
        It follows that $\Diff_{\anc{Q}{k}}b$ is an atomic block for every $k=1,\dots,|Q|$ and
$$
 \|\Diff_{\anc{Q}{k}}b\|_2  \lesssim \frac{\nu_{|x_Q|-k+1}^{|x_Q|}}{\mu(\anc{Q}{k-1})^{\frac{1}{2}}} +(q(x_Q^{(k)})-1)\frac{\nu_{|x_Q|-k}^{|x_Q|}}{\mu(\anc{Q}{k-1})^{\frac{1}{2}}}
\sim \frac{\nu_{|x_Q|-k+1}^{|x_Q|}}{\mu(\anc{Q}{k-1})^{\frac{1}{2}}}.
$$
Finally, we get
\begin{equation*}
    \|\Proj b\|_{\Hone}\lesssim 1+\sum_{k=1}^{|Q|}\nu_{|x_Q|-k+1}^{|x_Q|}\lesssim 1.
\end{equation*}
\end{proof}

\begin{remark}
    It is natural to ask whether the above estimates apply to the Bergman shifts introduced in Remark \ref{rem:remBergShifts}. It is not difficult to see that this is the case if $q$ is bounded, because in that case the corresponding kernels satisfy \eqref{eq:HormCond} and \eqref{eq:sizeCond}. However, in the general radial case $\mu$ fails to be balanced (in the sense of \cite{CPW2023}), so one cannot expect a general result to hold, and therefore there is no hope of a richer Calder\'on-Zygmund theory via representation as in \cite{Hy2012,DPWW2024}.
\end{remark}

\bibliographystyle{alpha}
\bibliography{sources}

\begin{thebibliography}{DPWW23}

\bibitem[BB78]{BB1978}
David B\'ekoll\'e and Aline Bonami.
\newblock In{\'e}galit{\'e}s {\`a} poids pour le noyau de {Bergman}.
\newblock {\em C. R. Acad. Sci., Paris, S{\'e}r. A}, 286:775--778, 1978.

\bibitem[CAP16]{CoPa2016}
Jos\'e~M. Conde-Alonso and Javier Parcet.
\newblock Atomic blocks for noncommutative martingales.
\newblock {\em Indiana Univ. Math. J.}, 65(4):1425--1443, 2016.

\bibitem[CAP19]{CoPa2019}
Jos\'e{}~M. Conde-Alonso and Javier Parcet.
\newblock Nondoubling {C}alder\'on-{Z}ygmund theory: a dyadic approach.
\newblock {\em J. Fourier Anal. Appl.}, 25(4):1267--1292, 2019.

\bibitem[CAPW24]{CPW2023}
Jos{\'e}~M. Conde~Alonso, Jill Pipher, and Nathan~A. Wagner.
\newblock Balanced measures, sparse domination and complexity-dependent weight
  classes.
\newblock {\em Math. Ann.}, 2024.

\bibitem[CC94]{CC1994}
Joel~M. Cohen and Flavia Colonna.
\newblock Embeddings of trees in the hyperbolic disk.
\newblock {\em Complex Variables Theory Appl.}, 24(3-4):311--335, 1994.

\bibitem[CCAP22]{CCP2022}
L\'eonard Cadilhac, Jos\'e{}~M. Conde-Alonso, and Javier Parcet.
\newblock Spectral multipliers in group algebras and noncommutative
  {C}alder\'on-{Z}ygmund theory.
\newblock {\em J. Math. Pures Appl. (9)}, 163:450--472, 2022.

\bibitem[CCPS16]{CCPS}
Joel~M. Cohen, Flavia Colonna, Massimo~A. Picardello, and David Singman.
\newblock Bergman spaces and {C}arleson measures on homogeneous isotropic
  trees.
\newblock {\em Potential Anal.}, 44(4):745--766, 2016.

\bibitem[CDPO18]{CuDPOu2018}
Amalia Culiuc, Francesco Di~Plinio, and Yumeng Ou.
\newblock Uniform sparse domination of singular integrals via dyadic shifts.
\newblock {\em Math. Res. Lett.}, 25(1):21--42, 2018.

\bibitem[CMM09]{CMM2009}
Andrea Carbonaro, Giancarlo Mauceri, and Stefano Meda.
\newblock {$H^1$} and {BMO} for certain locally doubling metric measure spaces.
\newblock {\em Ann. Sc. Norm. Super. Pisa Cl. Sci. (5)}, 8(3):543--582, 2009.

\bibitem[CMM10]{CMM2010}
Andrea Carbonaro, Giancarlo Mauceri, and Stefano Meda.
\newblock {$H^1$} and {BMO} for certain locally doubling metric measure spaces
  of finite measure.
\newblock {\em Colloq. Math.}, 118(1):13--41, 2010.

\bibitem[CUMP12]{CMP2012}
David Cruz-Uribe, Jos\'e{}~Mar\'ia Martell, and Carlos P\'erez.
\newblock Sharp weighted estimates for classical operators.
\newblock {\em Adv. Math.}, 229(1):408--441, 2012.

\bibitem[DMMR24]{dMMR2023}
Filippo De~Mari, Matteo Monti, and Elena Rizzo.
\newblock Horocyclic harmonic {B}ergman spaces on homogeneous trees.
\newblock {\em \text{\rm In print on} Analysis and Applications}, 2024.

\bibitem[DMMV24]{dMMV2023}
Filippo De~Mari, Matteo Monti, and Maria Vallarino.
\newblock Harmonic {B}ergman projectors on homogeneous trees.
\newblock {\em Potential Anal.}, 61(1):153--182, 2024.

\bibitem[DPWW23]{DPWW2024}
Francesco Di~Plinio, Brett~D. Wick, and Tyler Williams.
\newblock Wavelet representation of singular integral operators.
\newblock {\em Math. Ann.}, 386(3-4):1829--1889, 2023.

\bibitem[DV23]{DoVa2023}
Evgueni Doubtsov and Andrei~V. Vasin.
\newblock Calder\'on-{Z}ygmund operators on {RBMO}.
\newblock {\em Proc. Amer. Math. Soc.}, 151(2):595--610, 2023.

\bibitem[Hyt10]{Hy2010}
Tuomas Hyt\"onen.
\newblock A framework for non-homogeneous analysis on metric spaces, and the
  {RBMO} space of {T}olsa.
\newblock {\em Publ. Mat.}, 54(2):485--504, 2010.

\bibitem[Hyt12]{Hy2012}
Tuomas~P. Hyt\"onen.
\newblock The sharp weighted bound for general {C}alder\'on-{Z}ygmund
  operators.
\newblock {\em Ann. of Math. (2)}, 175(3):1473--1506, 2012.

\bibitem[KLW21]{KLW}
Matthias Keller, Daniel Lenz, and Rados\l aw~K. Wojciechowski.
\newblock {\em Graphs and discrete {D}irichlet spaces}, volume 358 of {\em
  Grundlehren der mathematischen Wissenschaften [Fundamental Principles of
  Mathematical Sciences]}.
\newblock Springer, Cham, [2021] \copyright 2021.

\bibitem[Ler13a]{Le2013}
Andrei~K. Lerner.
\newblock On an estimate of {C}alder\'on-{Z}ygmund operators by dyadic positive
  operators.
\newblock {\em J. Anal. Math.}, 121:141--161, 2013.

\bibitem[Ler13b]{Le2013b}
Andrei~K. Lerner.
\newblock A simple proof of the {$A_2$} conjecture.
\newblock {\em Int. Math. Res. Not. IMRN}, (14):3159--3170, 2013.

\bibitem[Lor21]{Lo2021}
Emiel Lorist.
\newblock On pointwise {$\ell^r$}-sparse domination in a space of homogeneous
  type.
\newblock {\em J. Geom. Anal.}, 31(9):9366--9405, 2021.

\bibitem[LSMP14]{LSMP2012}
Luis~Daniel L\'opez-S\'anchez, Jos\'e{}~Mar\'ia Martell, and Javier Parcet.
\newblock Dyadic harmonic analysis beyond doubling measures.
\newblock {\em Adv. Math.}, 267:44--93, 2014.

\bibitem[Mon]{Mo}
Matteo Monti.
\newblock ${H}^1$ and {BMO} spaces for exponentially decreasing measures on
  homogeneous trees.
\newblock to appear in: N. Arcozzi, M. M. Peloso and A. Tabacco, (New Trends
  in) Complex and Fourier Analysis, Springer INdAM Series, Springer. Cambridge
  Tracts in Mathematics, arXiv:2301.07600.

\bibitem[Per19]{Pe2019}
Mar\'ia~Cristina Pereyra.
\newblock Dyadic harmonic analysis and weighted inequalities: the sparse
  revolution.
\newblock In {\em New trends in applied harmonic analysis. {V}ol. 2---harmonic
  analysis, geometric measure theory, and applications}, Appl. Numer. Harmon.
  Anal., pages 159--239. Birkh\"auser/Springer, Cham, [2019] \copyright 2019.

\bibitem[Pet06]{Pet}
Peter Petersen.
\newblock {\em Riemannian geometry}, volume 171 of {\em Graduate Texts in
  Mathematics}.
\newblock Springer, New York, second edition, 2006.

\bibitem[PR13]{PR2013}
Sandra Pott and Maria~Carmen Reguera.
\newblock Sharp {B}\'ekoll\'e{} estimates for the {B}ergman projection.
\newblock {\em J. Funct. Anal.}, 265(12):3233--3244, 2013.

\bibitem[Tol01a]{To2001}
Xavier Tolsa.
\newblock B{MO}, {$H^1$}, and {C}alder\'on-{Z}ygmund operators for non doubling
  measures.
\newblock {\em Math. Ann.}, 319(1):89--149, 2001.

\bibitem[Tol01b]{To2001b}
Xavier Tolsa.
\newblock A proof of the weak {$(1,1)$} inequality for singular integrals with
  non doubling measures based on a {C}alder\'on-{Z}ygmund decomposition.
\newblock {\em Publ. Mat.}, 45(1):163--174, 2001.

\bibitem[VZK18]{VoZK2018}
Alexander Volberg and Pavel Zorin-Kranich.
\newblock Sparse domination on non-homogeneous spaces with an application to
  {$A_p$} weights.
\newblock {\em Rev. Mat. Iberoam.}, 34(3):1401--1414, 2018.

\bibitem[Woe00]{Wo}
Wolfgang Woess.
\newblock {\em Random walks on infinite graphs and groups}, volume 138 of {\em
  Cambridge Tracts in Mathematics}.
\newblock Cambridge University Press, Cambridge, 2000.

\end{thebibliography}

\noindent\textbf{Jose M. Conde-Alonso}  \\
\texttt{jose.conde@uam.es}      \\
Dept. of Mathematics, Universidad Aut\'onoma de Madrid, 7 Francisco Tom\'as y Valiente, 28049 Madrid, Spain. \\

\noindent\textbf{Filippo De Mari}  \\
\texttt{filippo.demari@unige.it}      \\
Dipartimento di Matematica and MaLGa Center, Università di Genova,
Via Dodecaneso 35, 16146 Genova, Italy. \\

\noindent\textbf{Matteo Monti}  \\
\texttt{matteo.monti@unibg.it}      \\
Dipartimento di Ingegneria Gestionale, dell’Informazione e della Produzione, Università di Bergamo,
Viale Marconi 5, 24044 Dalmine (BG), Italy.\\

\noindent\textbf{Elena Rizzo}  \\
\texttt{elena.rizzo@edu.unige.it}      \\
Dipartimento di Matematica and MaLGa Center, Università di Genova,
Via Dodecaneso 35, 16146 Genova, Italy.\\

\noindent\textbf{Maria Vallarino}  \\
\texttt{maria.vallarino@polito.it}      \\
Dipartimento di Scienze Matematiche ``Giuseppe Luigi Lagrange'', Politecnico di Torino, Corso Duca degli Abruzzi 24, 10129 Torino, Italy. \\ 

\end{document}